\documentclass[12pt]{article}

\usepackage{amsfonts,amsmath,amsthm,amssymb,amscd}
\usepackage[dvips]{graphicx}

\theoremstyle{definition}
\newtheorem{theorem}{Theorem}
\newtheorem{lemma}[theorem]{Lemma}
\newtheorem{proposition}[theorem]{Proposition}

\numberwithin{equation}{section}
\numberwithin{theorem}{section}

\begin{document}

\begin{center}
{\bf{\Large Note on some theorem of Farkas and Kra}}
\end{center}

\begin{center}
By Kazuhide Matsuda
\end{center}

\begin{center}
Faculty of Fundamental Science, National Institute of Technology, Niihama College,\\
7-1 Yagumo-chou, Niihama, Ehime, Japan, 792-8580. \\
E-mail: matsuda@sci.niihama-nct.ac.jp  \\
Fax: 81-0897-37-7809 
\end{center}

\noindent
{\bf Abstract}
In this paper, 
we apply high level versions of Jacobi's derivative formula to number theory 
such as quarternary quadratic forms and convolution sums of some arithmetical functions.

\section{Introduction}
\label{intro}
Throughout this paper, let $\mathbb{N}_0,$ $\mathbb{N},$ 
denote the set of nonnegative integers, positive integers, and 
denote the set of squares by $\{x^{2}: x \in \mathbb{Z}\}.$ A {\em triangular number} is defined as $t_x=x(x+1)/2$, where $x\in\mathbb{Z}.$ 
\par
For the positive integers $j,k,$ and $n\in\mathbb{N},$ 
$d_{j,k}(n)$ denotes the number of positive divisors $d$ of $n$ such that $d\equiv j \,\,\mathrm{mod} \,k,$ 
and 
$d_{j,k}^{*}(n)$ denotes the number of positive divisors $d$ of $n$ such that $d\equiv j \,\,\mathrm{mod} \,k$ and $n/d$ is odd,   
which implies that 
$
d_{j,k}^{*}(n)=d_{j,k}(n)-d_{j,k}(n/2). 
$
\par
For each $n\in\mathbb{N},$ 
$\sigma(n)$ is the sum of positive divisors of $n,$ 
and 
$\sigma^{*}(n)$ is the sum of positive divisors $d$ of $n$ such that 
$n/d$ is odd, 
which implies that 
$\sigma^{*}(n)=\sigma(n)-\sigma(n/2).$ 
Moreover, we set 
$d_{j,k}(n)=\sigma(n)=0$ for $n\in\mathbb{Q}\setminus\mathbb{N}_0.$  
\par
The {\it upper half plane} $\mathbb{H}^2$ is defined by 
$
\mathbb{H}^2=
\{
\tau\in\mathbb{C} \,\, | \,\, \Im \tau>0
\}, 
$
and the {\it Dedekind eta function} is defined by 
$
\displaystyle
\eta(\tau)=q^{\frac{1}{24}} \prod_{n=1}^{\infty} (1-q^n), \,\, q=\exp(2\pi i \tau) \,\, \text{for} \, \tau\in\mathbb{H}^2. 
$
\par
Following Farkas and Kra \cite{Farkas-Kra}, 
we introduce the {\it theta function with characteristics,} 
which is defined by 
\begin{align*}
\theta 
\left[
\begin{array}{c}
\epsilon \\
\epsilon^{\prime}
\end{array}
\right] (\zeta, \tau) 
=
\theta 
\left[
\begin{array}{c}
\epsilon \\
\epsilon^{\prime}
\end{array}
\right] (\zeta) 
:=&\sum_{n\in\mathbb{Z}} \exp
\left(2\pi i\left[ \frac12\left(n+\frac{\epsilon}{2}\right)^2 \tau+\left(n+\frac{\epsilon}{2}\right)\left(\zeta+\frac{\epsilon^{\prime}}{2}\right) \right] \right), 
\end{align*}
where $\epsilon, \epsilon^{\prime}\in\mathbb{R}, \, \zeta\in\mathbb{C},$ and $\tau\in\mathbb{H}^{2}.$ 
The {\it theta constants} are given by 
\begin{equation*}
\theta 
\left[
\begin{array}{c}
\epsilon \\
\epsilon^{\prime}
\end{array}
\right]
:=
\theta 
\left[
\begin{array}{c}
\epsilon \\
\epsilon^{\prime}
\end{array}
\right] (0, \tau).
\end{equation*}
Furthermore, 
we denote the derivative coefficients of the theta function by 
\begin{equation*}
\theta^{\prime} 
\left[
\begin{array}{c}
\epsilon \\
\epsilon^{\prime}
\end{array}
\right]
:=\left.
\frac{\partial}{\partial \zeta} 
\theta 
\left[
\begin{array}{c}
\epsilon \\
\epsilon^{\prime}
\end{array}
\right] (\zeta, \tau)
\right|_{\zeta=0}, 
\,
\theta^{\prime \prime} 
\left[
\begin{array}{c}
\epsilon \\
\epsilon^{\prime}
\end{array}
\right]
:=\left.
\frac{\partial^2 }{\partial \zeta^2} 
\theta 
\left[
\begin{array}{c}
\epsilon \\
\epsilon^{\prime}
\end{array}
\right] (\zeta, \tau)
\right|_{\zeta=0}.
\end{equation*}
\par
Farkas and Kra \cite{Farkas-Kra} 
treated the theta constants with {\it rational characteristics,} 
that is, 
the case where $\epsilon$ and $\epsilon^{\prime}$ are both rational numbers, 
and 
derived a number of interesting theta constant identities. 
For each positive integer $k,$ define the {\it Hecke group} $\Gamma_o(k)$ by 
\begin{equation*}
\Gamma_o(k)=
\left\{  
\begin{bmatrix}
a & b \\
c & d
\end{bmatrix}
\in
SL(2,\mathbb{Z});
c\equiv 0 \bmod k 
\right\}.
\end{equation*}
In \cite[pp. 318]{Farkas-Kra}, they proved the following result:
\begin{theorem}
(Farkas and Kra)
\label{thm:FrakasKra}
{\it
For each odd prime $k,$ 
\begin{equation*}
\frac{d}{d \tau} 
\log 
\left(
\frac{\eta(k\tau)}{\eta(\tau)}
\right)
+
\frac{1}{2\pi i (k-2)}
\sum_{l=0}^{\frac{k-3}{2}}
\left(
\frac
{
\theta^{\prime} 
\left[
\begin{array}{c}
1 \\
\frac{1+2l}{k}
\end{array}
\right](0,\tau) 
}
{
\theta
\left[
\begin{array}{c}
1 \\
\frac{1+2l}{k}
\end{array}
\right](0,\tau)
}
\right)^2, \quad \tau\in\mathbb{H}^2, 
\end{equation*}
is a cusp form for $\Gamma_{o}(k).$ 
This form is identically zero provided $k\leq 13,$ $k\neq 11.$ 
}
\end{theorem}
\par
The aim of this paper is to consider analogues of Theorem \ref{thm:FrakasKra} for $k=4,6,8$ 
and their applications to number theory such as quarternary quadratic forms and convolutions sums of some arihtmetical functions. 
For this purpose, 
we use the results of our papers \cite{Matsuda-1, Matsuda-2, Matsuda-3}, 
where we considered high level versions of 
{\it Jacobi's derivative formula}:  
\begin{equation}
\label{eqn:Jacobi-derivative}
\theta^{\prime} 
\left[
\begin{array}{c}
1 \\
1
\end{array}
\right] 
=
-\pi 
\theta
\left[
\begin{array}{c}
0 \\
0
\end{array}
\right] 
\theta
\left[
\begin{array}{c}
1 \\
0
\end{array}
\right] 
\theta
\left[
\begin{array}{c}
0 \\
1
\end{array}
\right].  
\end{equation}
\par
The remainder of this paper is as follows. 
In Sections 2 and 3, we review the basic properties and the drivative formulas of the theta functions, respectively.  
In Section 4, we deal with $\Gamma_{o}(4)$ and prove 4 squares theorem and 4 triangular numbers theorem. 
\par
In Section 5, we treat $\Gamma_{o}(6)$ and apply the results to the quadratic forms, 
\begin{equation*}
x^2+xy+y^2+z^2+zw+w^2, \,\, x^2+y^2+3z^2+3w^2, \,\,x^2+xy+y^2+2(z^2+zw+w^2),
\end{equation*}
and convolution sums of the arithmetical functions,
\begin{equation*}
\delta^{*}(n)=d^{*}_{1,3}(n)-d_{2,3}^{*}(n), \,\,
\epsilon^{*}(n)=d_{1,6}^{*}(n)+d_{2,6}^{*}(n)-d_{4,6}^{*}(n)-d_{5,6}^{*}(n) \,\,\text{for each $n\in\mathbb{N}$}. 
\end{equation*}
\par
In Section 6, we deal with  $\Gamma_{o}(8)$ and apply the result to the mixed sums of squares and triangular numbers,
\begin{equation*}
x^2+y^2+2z^2+2w^2, \,\,
x^2+2y^2+4t_z+4t_w, \,\,
x^2+y^2+4t_z+4t_w, \,\,
x^2+2y^2+2z^2+4t_w,
\end{equation*}
and
\begin{equation*}
x^2+y^2+z^2+2w^2, \,\,
x^2+2y^2+2z^2+2w^2, \,\,
x^2+y^2+z^2+4t_w, \,\,
x^2+4t_y+4t_z+4t_w.
\end{equation*}

\section{The properties of the theta functions}
\label{sec:properties}

\subsection{Basic properties}
We first note that 
for $m,n\in\mathbb{Z},$ 
\begin{equation}
\label{eqn:integer-char}
\theta 
\left[
\begin{array}{c}
\epsilon \\
\epsilon^{\prime}
\end{array}
\right] (\zeta+n+m\tau, \tau) =
\exp(2\pi i)\left[\frac{n\epsilon-m\epsilon^{\prime}}{2}-m\zeta-\frac{m^2\tau}{2}\right]
\theta 
\left[
\begin{array}{c}
\epsilon \\
\epsilon^{\prime}
\end{array}
\right] (\zeta,\tau),
\end{equation}
and 
\begin{equation}
\theta 
\left[
\begin{array}{c}
\epsilon +2m\\
\epsilon^{\prime}+2n
\end{array}
\right] 
(\zeta,\tau)
=\exp(\pi i \epsilon n)
\theta 
\left[
\begin{array}{c}
\epsilon \\
\epsilon^{\prime}
\end{array}
\right] 
(\zeta,\tau).
\end{equation}
Furthermore, 
it is easy to see that 
\begin{equation*}
\theta 
\left[
\begin{array}{c}
-\epsilon \\
-\epsilon^{\prime}
\end{array}
\right] (\zeta,\tau)
=
\theta 
\left[
\begin{array}{c}
\epsilon \\
\epsilon^{\prime}
\end{array}
\right] (-\zeta,\tau)
\,\,
\mathrm{and}
\,\,
\theta^{\prime} 
\left[
\begin{array}{c}
-\epsilon \\
-\epsilon^{\prime}
\end{array}
\right] (\zeta,\tau)
=
-
\theta^{\prime} 
\left[
\begin{array}{c}
\epsilon \\
\epsilon^{\prime}
\end{array}
\right] (-\zeta,\tau).
\end{equation*}
\par
For $m,n\in\mathbb{R},$ 
we see that 
\begin{align}
\label{eqn:real-char}
&\theta 
\left[
\begin{array}{c}
\epsilon \\
\epsilon^{\prime}
\end{array}
\right] \left(\zeta+\frac{n+m\tau}{2}, \tau\right)   \notag\\
&=
\exp(2\pi i)\left[
-\frac{m\zeta}{2}-\frac{m^2\tau}{8}-\frac{m(\epsilon^{\prime}+n)}{4}
\right]
\theta 
\left[
\begin{array}{c}
\epsilon+m \\
\epsilon^{\prime}+n
\end{array}
\right] 
(\zeta,\tau). 
\end{align}
We note that 
$\theta 
\left[
\begin{array}{c}
\epsilon \\
\epsilon^{\prime}
\end{array}
\right] \left(\zeta, \tau\right)$ has only one zero in the fundamental parallelogram, 
which is given by 
$$
\zeta=\frac{1-\epsilon}{2}\tau+\frac{1-\epsilon^{\prime}}{2}. 
$$

\subsection{Jacobi's triple product identity}
All the theta functions have infinite product expansions, which are given by 
\begin{align}
\theta 
\left[
\begin{array}{c}
\epsilon \\
\epsilon^{\prime}
\end{array}
\right] (\zeta, \tau) &=\exp\left(\frac{\pi i \epsilon \epsilon^{\prime}}{2}\right) x^{\frac{\epsilon^2}{4}} z^{\frac{\epsilon}{2}}    \notag  \\
                           &\quad 
                           \displaystyle \times\prod_{n=1}^{\infty}(1-x^{2n})(1+e^{\pi i \epsilon^{\prime}} x^{2n-1+\epsilon} z)(1+e^{-\pi i \epsilon^{\prime}} x^{2n-1-\epsilon}/z),  \label{eqn:Jacobi-triple}
\end{align}
where $x=\exp(\pi i \tau)$ and $z=\exp(2\pi i \zeta).$ 
Therefore, it follows from Jacobi's derivative formula (\ref{eqn:Jacobi-derivative}) that 
\begin{equation*}
\label{eqn:Jacobi}
\theta^{\prime} 
\left[
\begin{array}{c}
1 \\
1
\end{array}
\right](0,\tau) 
=
-2\pi 
q^{\frac18}
\prod_{n=1}^{\infty}(1-q^n)^3, \,\,q=\exp(2\pi i \tau). 
\end{equation*}

\subsection{Spaces of $N$-th order $\theta$-functions}

Following Farkas and Kra \cite{Farkas-Kra}, 
we define 
$\mathcal{F}_{N}\left[
\begin{array}{c}
\epsilon \\
\epsilon^{\prime}
\end{array}
\right] $ to be the set of entire functions $f$ that satisfy the two functional equations, 
$$
f(\zeta+1)=\exp(\pi i \epsilon) \,\,f(\zeta),
$$
and 
$$
f(\zeta+\tau)=\exp(-\pi i)[\epsilon^{\prime}+2N\zeta+N\tau] \,\,f(\zeta), \quad \zeta\in\mathbb{C},  \,\,\tau \in\mathbb{H}^2,
$$ 
where 
$N$ is a positive integer and 
$\left[
\begin{array}{c}
\epsilon \\
\epsilon^{\prime}
\end{array}
\right] \in\mathbb{R}^2.$ 
This set of functions is called the space of {\it $N$-th order $\theta$-functions with characteristic }
$\left[
\begin{array}{c}
\epsilon \\
\epsilon^{\prime}
\end{array}
\right]. $ 
Note that 
$$
\dim \mathcal{F}_{N}\left[
\begin{array}{c}
\epsilon \\
\epsilon^{\prime}
\end{array}
\right] =N.
$$
For its proof, see Farkas and Kra \cite[pp.133]{Farkas-Kra}.

\subsection{The heat equation}
The theta function satisfies the following heat equation: 
\begin{equation}
\label{eqn:heat}
\frac{\partial^2}{\partial \zeta^2}
\theta
\left[
\begin{array}{c}
\epsilon \\
\epsilon^{\prime}
\end{array}
\right](\zeta,\tau)
=
4\pi i
\frac{\partial}{\partial \tau}
\theta
\left[
\begin{array}{c}
\epsilon \\
\epsilon^{\prime}
\end{array}
\right](\zeta,\tau). 
\end{equation}

\subsection{Lemma of Farkas and Kra}

We recall the lemma of Farkas and Kra \cite[pp.78]{Farkas-Kra}.

\begin{lemma}
\label{lem:Farkas-Kra}
{\it
For 
all characteristics 
$
\left[
\begin{array}{c}
\epsilon \\
\epsilon^{\prime}
\end{array}
\right], 
\left[
\begin{array}{c}
\delta \\
\delta^{\prime}
\end{array}
\right]
$ 
and 
all $\tau\in\mathbb{H}^2,$ 
we have
\begin{align*}
&\theta
\left[
\begin{array}{c}
\epsilon \\
\epsilon^{\prime}
\end{array}
\right](0,\tau)
\theta
\left[
\begin{array}{c}
\delta \\
\delta^{\prime}
\end{array}
\right](0,\tau)  \\
=&
\theta
\left[
\begin{array}{c}
\frac{\epsilon+\delta}{2} \\
\epsilon^{\prime}+\delta^{\prime}
\end{array}
\right](0,2\tau)
\theta
\left[
\begin{array}{c}
\frac{\epsilon-\delta}{2} \\
\epsilon^{\prime}-\delta^{\prime}
\end{array}
\right](0,2\tau)  
+
\theta
\left[
\begin{array}{c}
\frac{\epsilon+\delta}{2}+1 \\
\epsilon^{\prime}+\delta^{\prime}
\end{array}
\right](0,2\tau)
\theta
\left[
\begin{array}{c}
\frac{\epsilon-\delta}{2}+1 \\
\epsilon^{\prime}-\delta^{\prime}
\end{array}
\right](0,2\tau). 
\end{align*}
}
\end{lemma}

\section{Derivative formulas}

From Matsuda \cite{Matsuda-1, Matsuda-2, Matsuda-3}, 
recall the following derivative formulas:

\begin{theorem}
\label{thm:analogue-Jacobi-derivative}
{\it
For every $\tau\in\mathbb{H}^2,$ we have
\begin{equation}
\label{eqn:analogue-1-1/2}  
\theta^{\prime}
\left[
\begin{array}{c}
1 \\
\frac12
\end{array}
\right](0,\tau)
=
-\pi
\theta^{2}
\left[
\begin{array}{c}
0 \\
0
\end{array}
\right](0,2\tau)
\theta
\left[
\begin{array}{c}
1 \\
\frac12
\end{array}
\right](0,\tau), 
\end{equation}
\begin{equation}
\label{eqn:analogue-0-1/2}
\theta^{\prime}
\left[
\begin{array}{c}
0 \\
\frac12
\end{array}
\right](0,\tau)
=
-\pi
\theta^{2}
\left[
\begin{array}{c}
1 \\
0
\end{array}
\right](0,2\tau)
\theta
\left[
\begin{array}{c}
0 \\
\frac12
\end{array}
\right](0,\tau),  
\end{equation}
\begin{equation}
\label{eqn:analogue-1-1/3,2/3}
\frac
{
\theta^{\prime}
\left[
\begin{array}{c}
1 \\
\frac13
\end{array}
\right]
}
{
\theta
\left[
\begin{array}{c}
1 \\
\frac13
\end{array}
\right]
}
=
\frac16
\theta^{\prime}
\left[
\begin{array}{c}
1 \\
1
\end{array}
\right]
\frac
{
\left(
\theta^4
\left[
\begin{array}{c}
1 \\
\frac13
\end{array}
\right]
-
3
\theta^4
\left[
\begin{array}{c}
1 \\
\frac23
\end{array}
\right]
\right)
}
{
\theta
\left[
\begin{array}{c}
1 \\
0
\end{array}
\right]
\theta
\left[
\begin{array}{c}
1 \\
\frac13
\end{array}
\right]
\theta^3
\left[
\begin{array}{c}
1 \\
\frac23
\end{array}
\right]
}, \quad
\frac
{
\theta^{\prime}
\left[
\begin{array}{c}
1 \\
\frac23
\end{array}
\right]
}
{
\theta
\left[
\begin{array}{c}
1 \\
\frac23
\end{array}
\right]
}
=
\frac13
\theta^{\prime}
\left[
\begin{array}{c}
1 \\
1
\end{array}
\right]
\frac
{
\theta^4
\left[
\begin{array}{c}
1 \\
\frac13
\end{array}
\right]
}
{
\theta
\left[
\begin{array}{c}
1 \\
0
\end{array}
\right]
\theta
\left[
\begin{array}{c}
1 \\
\frac13
\end{array}
\right]
\theta^3
\left[
\begin{array}{c}
1 \\
\frac23
\end{array}
\right]
}, 
\end{equation}
\begin{equation}
\label{eqn:analogue-1-1/4}
\theta^{\prime}
\left[
\begin{array}{c}
1 \\
\frac14
\end{array}
\right](0,\tau)
=
-
\pi
\theta
\left[
\begin{array}{c}
1 \\
\frac14
\end{array}
\right](0,\tau)
\theta
\left[
\begin{array}{c}
0 \\
0
\end{array}
\right](0,4\tau)
\left\{
\sqrt{2}
\theta
\left[
\begin{array}{c}
0 \\
0
\end{array}
\right](0,2\tau)
-
\theta
\left[
\begin{array}{c}
0 \\
0
\end{array}
\right](0,4\tau)
\right\},  
\end{equation}
\begin{equation}
\label{eqn:analogue-1-3/4}
\theta^{\prime}
\left[
\begin{array}{c}
1 \\
\frac34
\end{array}
\right](0,\tau)
=
-
\pi
\theta
\left[
\begin{array}{c}
1 \\
\frac34
\end{array}
\right](0,\tau)
\theta
\left[
\begin{array}{c}
0 \\
0
\end{array}
\right](0,4\tau)
\left\{
\sqrt{2}
\theta
\left[
\begin{array}{c}
0 \\
0
\end{array}
\right](0,2\tau)
+
\theta
\left[
\begin{array}{c}
0 \\
0
\end{array}
\right](0,4\tau)
\right\},
\end{equation}
\begin{equation}
\label{eqn:analogue-0-1/4}
\theta^{\prime}
\left[
\begin{array}{c}
0 \\
\frac14
\end{array}
\right](0,\tau)
=
-
\pi
\theta
\left[
\begin{array}{c}
0 \\
\frac14
\end{array}
\right](0,\tau)
\theta
\left[
\begin{array}{c}
1 \\
0
\end{array}
\right](0,4\tau)
\left\{
\sqrt{2}
\theta
\left[
\begin{array}{c}
0 \\
0
\end{array}
\right](0,2\tau)
-
\theta
\left[
\begin{array}{c}
1 \\
0
\end{array}
\right](0,4\tau)
\right\},  
\end{equation}
\begin{equation}
\label{eqn:analogue-0-3/4}
\theta^{\prime}
\left[
\begin{array}{c}
0 \\
\frac34
\end{array}
\right](0,\tau)
=
-
\pi
\theta
\left[
\begin{array}{c}
0 \\
\frac34
\end{array}
\right](0,\tau)
\theta
\left[
\begin{array}{c}
1 \\
0
\end{array}
\right](0,4\tau)
\left\{
\sqrt{2}
\theta
\left[
\begin{array}{c}
0 \\
0
\end{array}
\right](0,2\tau)
+
\theta
\left[
\begin{array}{c}
1 \\
0
\end{array}
\right](0,4\tau)
\right\}.
\end{equation}
}
\end{theorem}

\section{Case for $\Gamma_{o}(4)$}

\subsection{Case where $\epsilon=1, \epsilon^{\prime}=1/2$}

\subsubsection{Theorem for $\epsilon=1, \epsilon^{\prime}=1/2$}

\begin{theorem}
(Farkas and Kra \cite[pp. 321]{Farkas-Kra})
\label{thm:1-1/2}
{\it
For every $\tau\in\mathbb{H}^2,$ 
we have 
\begin{equation*}
\frac{d}{d\tau} 
\log \frac{\eta(4\tau)}{\eta(\tau)}
+
\frac{1}{2\pi i \cdot 2} 
\left\{
\frac
{
\theta^{\prime} 
\left[
\begin{array}{c}
1 \\
\frac12
\end{array}
\right](0,\tau) 
}
{
\theta
\left[
\begin{array}{c}
1 \\
\frac12
\end{array}
\right](0,\tau)
}
\right\}^2=0. 
\end{equation*}
}
\end{theorem}

\begin{proof}
Consider the following elliptic function:
\begin{equation*}
\varphi(z)
=
\frac
{
\theta
\left[
\begin{array}{c}
1 \\
0
\end{array}
\right](z)
\theta^2
\left[
\begin{array}{c}
1 \\
\frac12
\end{array}
\right](z) 
}
{
\theta^3
\left[
\begin{array}{c}
1 \\
1
\end{array}
\right](z)
}. 
\end{equation*}
In the fundamental parallelogram, 
the pole of $\varphi(z)$ is $z=0,$ 
which implies that 
$\mathrm{Res} (\varphi(z), 0)=0.$ 
Therefore, it follows that 
\begin{equation*}
\frac
{
\theta^{\prime\prime}
\left[
\begin{array}{c}
1 \\
0
\end{array}
\right]
}
{
\theta
\left[
\begin{array}{c}
1 \\
0
\end{array}
\right]
}
+
2
\frac{
\theta^{\prime\prime}
\left[
\begin{array}{c}
1 \\
\frac12
\end{array}
\right]
}
{
\theta
\left[
\begin{array}{c}
1 \\
\frac12
\end{array}
\right]
}
-
\frac
{
\theta^{\prime\prime \prime}
\left[
\begin{array}{c}
1 \\
1
\end{array}
\right]
}
{
\theta^{\prime}
\left[
\begin{array}{c}
1 \\
1
\end{array}
\right]
}
+
2
\left\{
\frac
{
\theta^{\prime} 
\left[
\begin{array}{c}
1 \\
\frac12
\end{array}
\right]
}
{
\theta
\left[
\begin{array}{c}
1 \\
\frac12
\end{array}
\right]
}
\right\}^2=0. 
\end{equation*}
The heat equation (\ref{eqn:heat}) implies that 
\begin{equation*}
4\pi i 
\frac{d}{d\tau}
\log 
\frac
{
\theta
\left[
\begin{array}{c}
1 \\
0
\end{array}
\right]
\theta^2
\left[
\begin{array}{c}
1 \\
\frac12
\end{array}
\right]
}
{
\theta^{\prime}
\left[
\begin{array}{c}
1 \\
1
\end{array}
\right]
}
+
2
\left\{
\frac
{
\theta^{\prime} 
\left[
\begin{array}{c}
1 \\
\frac12
\end{array}
\right]
}
{
\theta
\left[
\begin{array}{c}
1 \\
\frac12
\end{array}
\right]
}
\right\}^2=0. 
\end{equation*}
The theorem follows from Jacobi's triple product identity (\ref{eqn:Jacobi-triple}). 
\end{proof}

\subsubsection{4 squares theorem}

\begin{theorem}
\label{thm:4-squares}
{\it
For each $n\in\mathbb{N},$ 
\begin{equation*}
S_4(n)=:\sharp
\left\{
(x,y,z,w)\in\mathbb{Z}^4 \, | \, 
x^2+y^2+z^2+w^2=n
\right\}.
\end{equation*}
Then, 
we have 
\begin{equation*}
S_4(n)=
8\sigma(n)-32\sigma(n/4). 
\end{equation*}
}
\end{theorem}

\begin{proof}
Set $q=\exp(2\pi i \tau).$ 
The derivative formula (\ref{eqn:analogue-1-1/2}) and Theorem \ref{thm:1-1/2} imply that 
\begin{equation*}
\frac{d}{d\tau} 
\log \frac{\eta(4\tau)}{\eta(\tau)}
+
\frac{1}{2\pi i \cdot 2} 
\left\{
-\pi
\theta^2
\left[
\begin{array}{c}
0 \\
0
\end{array}
\right](0,2\tau)
\right\}^2=0,
\end{equation*}
which shows that 
\begin{equation*}
1+
\sum_{n=1}^{\infty} S_4(n) q^n
=
1+
8\sum_{n=1}^{\infty} \left( \sigma(n)-4 \sigma(n/4) \right) q^n. 
\end{equation*}
\end{proof}

\subsection{Case where $\epsilon=0, \epsilon^{\prime}=1/2$}

\subsubsection{Theorem for $\epsilon=0, \epsilon^{\prime}=1/2$}

\begin{theorem}
(Farkas and Kra \cite[pp. 321]{Farkas-Kra})
\label{thm:0-1/2}
{\it
For every $\tau\in\mathbb{H}^2,$ 
we have 
\begin{equation*}
\frac{d}{d\tau} 
\log \frac{\eta^3(2\tau)}{\eta^2(\tau) \eta(4\tau)}
+
\frac{1}{2\pi i \cdot 2} 
\left\{
\frac
{
\theta^{\prime} 
\left[
\begin{array}{c}
0 \\
\frac12
\end{array}
\right](0,\tau) 
}
{
\theta
\left[
\begin{array}{c}
0 \\
\frac12
\end{array}
\right](0,\tau)
}
\right\}^2=0. 
\end{equation*}
}
\end{theorem}

\begin{proof}
Consider the following elliptic function:
\begin{equation*}
\psi(z)
=
\frac
{
\theta
\left[
\begin{array}{c}
1 \\
0
\end{array}
\right](z)
\theta^2
\left[
\begin{array}{c}
0 \\
\frac12
\end{array}
\right](z) 
}
{
\theta^3
\left[
\begin{array}{c}
1 \\
1
\end{array}
\right](z)
}. 
\end{equation*}
In the fundamental parallelogram, 
the pole of $\psi(z)$ is $z=0,$ 
which implies that 
$\mathrm{Res} (\psi(z), 0)=0.$ 
Therefore, it follows that 
\begin{equation*}
\frac
{
\theta^{\prime\prime}
\left[
\begin{array}{c}
1 \\
0
\end{array}
\right]
}
{
\theta
\left[
\begin{array}{c}
1 \\
0
\end{array}
\right]
}
+
2
\frac{
\theta^{\prime\prime}
\left[
\begin{array}{c}
0 \\
\frac12
\end{array}
\right]
}
{
\theta
\left[
\begin{array}{c}
0 \\
\frac12
\end{array}
\right]
}
-
\frac
{
\theta^{\prime\prime \prime}
\left[
\begin{array}{c}
1 \\
1
\end{array}
\right]
}
{
\theta^{\prime}
\left[
\begin{array}{c}
1 \\
1
\end{array}
\right]
}
+
2
\left\{
\frac
{
\theta^{\prime} 
\left[
\begin{array}{c}
0 \\
\frac12
\end{array}
\right]
}
{
\theta
\left[
\begin{array}{c}
0 \\
\frac12
\end{array}
\right]
}
\right\}^2=0. 
\end{equation*}
The heat equation (\ref{eqn:heat}) implies that 
\begin{equation*}
4\pi i 
\frac{d}{d\tau}
\log 
\frac
{
\theta
\left[
\begin{array}{c}
1 \\
0
\end{array}
\right]
\theta^2
\left[
\begin{array}{c}
0 \\
\frac12
\end{array}
\right]
}
{
\theta^{\prime}
\left[
\begin{array}{c}
1 \\
1
\end{array}
\right]
}
+
2
\left\{
\frac
{
\theta^{\prime} 
\left[
\begin{array}{c}
0 \\
\frac12
\end{array}
\right]
}
{
\theta
\left[
\begin{array}{c}
0 \\
\frac12
\end{array}
\right]
}
\right\}^2=0. 
\end{equation*}
The theorem follows from Jacobi's triple product identity (\ref{eqn:Jacobi-triple}). 
\end{proof}

\subsubsection{4 triangular numbers theorem}

\begin{theorem}
\label{thm:4-triangular}
{\it
For each $n\in\mathbb{N}_0,$ set 
\begin{equation*}
T_4(n):=
\sharp
\left\{
(x,y,z,w)\in\mathbb{Z}^4 \, | \,
t_x+t_y+t_z+t_w=n
\right\}. 
\end{equation*}
Then, we have 
\begin{equation*}
T_4(n)=16\sigma(2n+1). 
\end{equation*}
}
\end{theorem}

\begin{proof}
Set $q=\exp(2\pi i \tau).$ 
The derivative formula (\ref{eqn:analogue-0-1/2}) and Theorem \ref{thm:0-1/2} imply that 
\begin{equation*}
\frac{d}{d\tau} 
\log \frac{\eta^3(2\tau)}{\eta^2(\tau) \eta(4\tau)}
+
\frac{1}{2\pi i \cdot 2} 
\left\{
-\pi
\theta^2
\left[
\begin{array}{c}
1 \\
0
\end{array}
\right](0,2\tau)
\right\}^2=0,
\end{equation*}
which shows that 
\begin{align*}
\sum_{n=0}^{\infty} T_4(n) q^{2n+1}
=&
16\sum_{n=1}^{\infty} \left( \sigma(n)-3 \sigma(n/2) + 2 \sigma(n/4)\right) q^n  \\
=&
16\sum_{n=0}^{\infty} \sigma(2n+1) q^{2n+1}. 
\end{align*}
\end{proof}

\section{Case for $\Gamma_{o}(6)$}
In this section, we introduce the following formula:
\begin{equation}
\label{eqn:formulra-a(q)}
a(q)=:\sum_{m,n\in\mathbb{Z}} q^{m^2+mn+n^2}
=
1+6\sum_{n=1}^{\infty}(d_{1,3}(n)-d_{2,3}(n)) q^n \,\,\mathrm{for} \,\,q\in\mathbb{C} \,\,\mathrm{with} \,\,|q|<1,
\end{equation}
which is equivalent to the following formula: for each $n\in\mathbb{N},$ 
\begin{equation}
\label{eqn:x^2;xy+y^2}
\sharp
\left\{
(x,y)\in\mathbb{Z}^2 \, | \,
x^2+xy+y^2
\right\}
=
6(d_{1,3}(n)-d_{2,3}(n)).
\end{equation}
For the proof of the formula (\ref{eqn:formulra-a(q)}), see Berndt \cite[pp. 79]{Berndt}. 
For the elementary proof of the formula (\ref{eqn:x^2;xy+y^2}), see Dickson \cite[pp. 68]{Dickson}.

\subsection{Preliminary results}

\begin{proposition}
\label{prop:preliminary-1/3-2/3}
{\it
For every $\tau\in\mathbb{H}^2,$ set $q=\exp(2\pi i \tau).$ 
Then, for every $\tau\in\mathbb{H}^2$ and $j=1,2,$ we have 
\begin{equation*}
\frac
{
\theta^{\prime} 
\left[
\begin{array}{c}
1 \\
\frac13
\end{array}
\right](0,\tau) 
}
{
\theta
\left[
\begin{array}{c}
1 \\
\frac13
\end{array}
\right](0,\tau)
}
=
-
\frac{\pi}{\sqrt{3}} a(q),   \,\,
\frac
{
\theta^{\prime} 
\left[
\begin{array}{c}
1 \\
\frac13
\end{array}
\right](0,\tau) 
}
{
\theta
\left[
\begin{array}{c}
1 \\
\frac13
\end{array}
\right](0,\tau)
}
-
\frac
{
\theta^{\prime} 
\left[
\begin{array}{c}
1 \\
\frac23
\end{array}
\right](0,\tau) 
}
{
\theta
\left[
\begin{array}{c}
1 \\
\frac23
\end{array}
\right](0,\tau)
}
=
-2
\frac
{
\theta^{\prime} 
\left[
\begin{array}{c}
1 \\
\frac13
\end{array}
\right](0,2\tau) 
}
{
\theta
\left[
\begin{array}{c}
1 \\
\frac13
\end{array}
\right](0,2\tau)
},
\end{equation*}
\begin{equation*}
\frac
{
\theta^{\prime} 
\left[
\begin{array}{c}
0 \\
\frac{j}{3}
\end{array}
\right](0,\tau) 
}
{
\theta
\left[
\begin{array}{c}
0 \\
\frac{j}{3}
\end{array}
\right](0,\tau)
}
=
\frac
{
\theta^{\prime} 
\left[
\begin{array}{c}
1 \\
\frac{j}{3}
\end{array}
\right](0,\tau/2) 
}
{
\theta
\left[
\begin{array}{c}
1 \\
\frac{j}{3}
\end{array}
\right](0,\tau/2)
}
-
\frac
{
\theta^{\prime} 
\left[
\begin{array}{c}
1 \\
\frac{j}{3}
\end{array}
\right](0,\tau) 
}
{
\theta
\left[
\begin{array}{c}
1 \\
\frac{j}{3}
\end{array}
\right](0,\tau)
}. 
\end{equation*}
}
\end{proposition}

\begin{proof}
Set $x=\exp(\pi i \tau).$ 
Jacobi's triple product identity (\ref{eqn:Jacobi-triple}) yields 
\begin{equation*}
\frac
{
\theta^{\prime} 
\left[
\begin{array}{c}
1 \\
\frac13
\end{array}
\right]
}
{
\theta
\left[
\begin{array}{c}
1 \\
\frac13
\end{array}
\right]
}
=
-
\frac{\pi}{\sqrt{3}}
\left\{
1+ 
6\sum_{n=1}^{\infty}(d_{1,3}(n)-d_{2,3}(n)) q^n
\right\},
\end{equation*}
\begin{equation*}
\frac
{
\theta^{\prime} 
\left[
\begin{array}{c}
1 \\
\frac23
\end{array}
\right]
}
{
\theta
\left[
\begin{array}{c}
1 \\
\frac23
\end{array}
\right]
}
=
-
\sqrt{3}\pi 
\left\{
1+2\sum_{n=1}^{\infty}(d_{1,6}(n)+d_{2,6}(n)-d_{4,6}(n)-d_{5,6}(n)) q^n,
\right\}. 
\end{equation*}
\begin{equation*}
\frac
{
\theta^{\prime} 
\left[
\begin{array}{c}
0 \\
\frac13
\end{array}
\right]
}
{
\theta
\left[
\begin{array}{c}
0 \\
\frac13
\end{array}
\right]
}
=
-2\sqrt{3}\pi \sum_{n=1}^{\infty}(d^{*}_{1,3}(n)-d_{2,3}^{*}(n)) x^n,
\end{equation*}
and
\begin{equation*}
\frac
{
\theta^{\prime} 
\left[
\begin{array}{c}
0 \\
\frac23
\end{array}
\right]
}
{
\theta
\left[
\begin{array}{c}
0 \\
\frac23
\end{array}
\right]
}
=
-2\sqrt{3}\pi \sum_{n=1}^{\infty}(d^{*}_{1,6}(n)+d_{2,6}^{*}(n)-d_{4,6}^{*}(n)-d^{*}_{5,6}(n)) x^n,
\end{equation*}
which prove the proposition. 
\end{proof}

\subsection{Theorem for $\epsilon=1, \epsilon^{\prime}=1/3$}

\subsubsection{Theorem for $\epsilon=1, \epsilon^{\prime}=1/3$}

\begin{theorem}
(Farkas and Kra \cite[pp. 318]{Farkas-Kra})
\label{thm:1-1/3}
{\it
For every $\tau\in\mathbb{H}^2,$ 
we have 
\begin{equation*}
\frac{d}{d\tau} 
\log \frac{\eta(3\tau)}{\eta(\tau)}
+
\frac{1}{2\pi i } 
\left\{
\frac
{
\theta^{\prime} 
\left[
\begin{array}{c}
1 \\
\frac13
\end{array}
\right](0,\tau) 
}
{
\theta
\left[
\begin{array}{c}
1 \\
\frac13
\end{array}
\right](0,\tau)
}
\right\}^2=0. 
\end{equation*}
}
\end{theorem}

\begin{proof}
Consider the following elliptic function:
\begin{equation*}
\varphi(z)
=
\frac
{
\theta^3
\left[
\begin{array}{c}
1 \\
\frac13
\end{array}
\right](z) 
}
{
\theta^3
\left[
\begin{array}{c}
1 \\
1
\end{array}
\right](z)
}. 
\end{equation*}
In the fundamental parallelogram, 
the pole of $\varphi(z)$ is $z=0,$ 
which implies that 
$\mathrm{Res} (\varphi(z), 0)=0.$ 
Therefore, it follows that 
\begin{equation}
\label{eqn:relation-for-thm-1-1/3}
3
\frac{
\theta^{\prime\prime}
\left[
\begin{array}{c}
1 \\
\frac13
\end{array}
\right]
}
{
\theta
\left[
\begin{array}{c}
1 \\
\frac13
\end{array}
\right]
}
-
\frac
{
\theta^{\prime\prime \prime}
\left[
\begin{array}{c}
1 \\
1
\end{array}
\right]
}
{
\theta^{\prime}
\left[
\begin{array}{c}
1 \\
1
\end{array}
\right]
}
+
6
\left\{
\frac
{
\theta^{\prime} 
\left[
\begin{array}{c}
1 \\
\frac13
\end{array}
\right]
}
{
\theta
\left[
\begin{array}{c}
1 \\
\frac13
\end{array}
\right]
}
\right\}^2=0. 
\end{equation}
The heat equation (\ref{eqn:heat}) implies that 
\begin{equation*}
4\pi i 
\frac{d}{d\tau}
\log 
\frac
{
\theta^3
\left[
\begin{array}{c}
1 \\
\frac13
\end{array}
\right]
}
{
\theta^{\prime}
\left[
\begin{array}{c}
1 \\
1
\end{array}
\right]
}
+
6
\left\{
\frac
{
\theta^{\prime} 
\left[
\begin{array}{c}
1 \\
\frac13
\end{array}
\right]
}
{
\theta
\left[
\begin{array}{c}
1 \\
\frac13
\end{array}
\right]
}
\right\}^2=0. 
\end{equation*}
The theorem follows from Jacobi's triple product identity (\ref{eqn:Jacobi-triple}). 
\end{proof}

\subsubsection{On $x^2+xy+y^2+z^2+zw+w^2$}

\begin{theorem}
\label{thm:x^2+xy+y^2-(1,1)}
{\it
For each $n\in\mathbb{N},$ set 
\begin{equation*}
s_2(n):=
\sharp
\left\{
(x,y,z,w)\in\mathbb{Z}^4 \, | \,
x^2+xy+y^2+z^2+zw+w^2=n
\right\}. 
\end{equation*}
Then, we have 
\begin{equation*}
s_2(n)=12\sigma(n)-36\sigma(n/3). 
\end{equation*}
}
\end{theorem}

\begin{proof}
The theorem follows from Proposition \ref{prop:preliminary-1/3-2/3} and Theorem \ref{thm:1-1/3}. 
\end{proof}

\subsubsection{On $x^2+y^2+3z^2+3w^2$}

\begin{theorem}
\label{thm:1133}
{\it
For each $n\in\mathbb{N},$ 
set
\begin{equation*}
S_{1,1,3,3}(n)
:=
\sharp
\left\{
(x,y,z,w)\in\mathbb{Z}^4 \, | \, 
x^2+y^2+3z^2+3w^2=n
\right\}. 
\end{equation*}
Then, 
\begin{equation*}
S_{1,1,3,3}(n)=4(-1)^{n-1}(\sigma(n)-4\sigma(n/2)-3\sigma(n/3)+12\sigma(n/6)).
\end{equation*}
}
\end{theorem}

\begin{proof}
Consider the following elliptic function:
\begin{equation*}
\psi(z)
=
\frac
{
\theta^2
\left[
\begin{array}{c}
1 \\
\frac23
\end{array}
\right](z)
\theta
\left[
\begin{array}{c}
1 \\
-\frac13
\end{array}
\right](z) 
}
{
\theta^3
\left[
\begin{array}{c}
1 \\
1
\end{array}
\right](z)
}. 
\end{equation*}
In the fundamental parallelogram, 
the pole of $\psi(z)$ is $z=0,$ 
which implies that 
$\mathrm{Res} (\psi(z), 0)=0.$ 
Therefore, it follows that 
\begin{equation}
\label{eqn:relation-1133-(1)}
\frac
{
\theta^{\prime\prime}
\left[
\begin{array}{c}
1 \\
\frac13
\end{array}
\right]
}
{
\theta
\left[
\begin{array}{c}
1 \\
\frac13
\end{array}
\right]
}
+
2
\frac
{
\theta^{\prime\prime}
\left[
\begin{array}{c}
1 \\
\frac23
\end{array}
\right]
}
{
\theta
\left[
\begin{array}{c}
1 \\
\frac23
\end{array}
\right]
}
-
\frac
{
\theta^{\prime\prime \prime}
\left[
\begin{array}{c}
1 \\
1
\end{array}
\right]
}
{
\theta^{\prime}
\left[
\begin{array}{c}
1 \\
1
\end{array}
\right]
}
-
4
\frac
{
\theta^{\prime} 
\left[
\begin{array}{c}
1 \\
\frac13
\end{array}
\right]
}
{
\theta
\left[
\begin{array}{c}
1 \\
\frac13
\end{array}
\right]
}
\cdot
\frac
{
\theta^{\prime} 
\left[
\begin{array}{c}
1 \\
\frac23
\end{array}
\right]
}
{
\theta
\left[
\begin{array}{c}
1 \\
\frac23
\end{array}
\right]
}
+
2
\left\{
\frac
{
\theta^{\prime} 
\left[
\begin{array}{c}
1 \\
\frac23
\end{array}
\right]
}
{
\theta
\left[
\begin{array}{c}
1 \\
\frac23
\end{array}
\right]
}
\right\}^2=0. 
\end{equation}
The heat equation (\ref{eqn:heat}) and the derivative formulas (\ref{eqn:analogue-1-1/3,2/3}) imply that 
\begin{equation}
\label{eqn:relation-1133-(2)}
4\pi i 
\frac{d}{d\tau}
\log 
\frac
{
\theta
\left[
\begin{array}{c}
1 \\
\frac13
\end{array}
\right]
\theta^2
\left[
\begin{array}{c}
1 \\
\frac23
\end{array}
\right]
}
{
\theta^{\prime}
\left[
\begin{array}{c}
1 \\
1
\end{array}
\right]
}
+
\frac23
\frac
{
\left\{
\theta^{\prime} 
\left[
\begin{array}{c}
1\\
1
\end{array}
\right]
\right\}^2
\theta^2
\left[
\begin{array}{c}
1 \\
\frac13
\end{array}
\right]
}
{
\theta^2
\left[
\begin{array}{c}
1 \\
0
\end{array}
\right]
\theta^2
\left[
\begin{array}{c}
1 \\
\frac23
\end{array}
\right]
}
=0. 
\end{equation}
Jacobi's triple product identity (\ref{eqn:Jacobi-triple}) yields 
\begin{equation*}
\frac
{
\left\{
\theta^{\prime} 
\left[
\begin{array}{c}
1\\
1
\end{array}
\right]
\right\}^2
\theta^2
\left[
\begin{array}{c}
1 \\
\frac13
\end{array}
\right]
}
{
\theta^2
\left[
\begin{array}{c}
1 \\
0
\end{array}
\right]
\theta^2
\left[
\begin{array}{c}
1 \\
\frac23
\end{array}
\right]
}
=
\left\{
-\sqrt{3}\pi 
\frac
{
\eta^2(\tau) \eta^2(3\tau)
}
{
\eta(2\tau) \eta(6\tau) 
}
\right\}^2, 
\end{equation*}
\begin{equation*}
\left\{
\frac
{
\eta^2(\tau) \eta^2(3\tau)
}
{
\eta(2\tau) \eta(6\tau) 
}
\right\}^2
=
1-4
\sum_{n=1}^{\infty} \left( \sigma(n)-4\sigma(n/2)-3\sigma(n/3)+12\sigma(n/6) \right) q^n, q=\exp(2\pi i \tau),
\end{equation*} 
and
\begin{equation*}
\theta^2
\left[
\begin{array}{c}
0\\
1
\end{array}
\right](0,2\tau)
\theta^2
\left[
\begin{array}{c}
0\\
1
\end{array}
\right](0,6\tau)
=
\left\{
\left(
\sum_{m\in\mathbb{Z}}
(-1)^m q^{m^2}
\right)
\left(
\sum_{n\in\mathbb{Z}}
(-1)^n q^{3n^2}
\right)
\right\}^2
=
\left\{
\frac
{
\eta^2(\tau) \eta^2(3\tau)
}
{
\eta(2\tau) \eta(6\tau) 
}
\right\}^2,
\end{equation*}
which imply that 
\begin{equation*}
\sum_{n=0}^{\infty}(-1)^n S_{1,1,3,3}(n) q^n
=
1-4
\sum_{n=1}^{\infty} \left(  \sigma(n)-4\sigma(n/2)-3\sigma(n/3)+12\sigma(n/6) \right) q^n. 
\end{equation*}
\end{proof}

\subsection{Theorem for $\epsilon=1, \epsilon^{\prime}=2/3$}

\begin{proposition}
\label{prop:theta-function-1-1/3,2/3}
{\it
For every $(z,\tau)\in\mathbb{C}\times \mathbb{H}^2,$ 
we have 
\begin{align}
&
\theta^2
\left[
\begin{array}{c}
1 \\
\frac23
\end{array}
\right]
\theta^2
\left[
\begin{array}{c}
1 \\
0
\end{array}
\right](z,\tau)
+
\theta^2
\left[
\begin{array}{c}
1 \\
0
\end{array}
\right]
\theta 
\left[
\begin{array}{c}
1 \\
\frac23
\end{array}
\right](z,\tau)
\theta
\left[
\begin{array}{c}
1 \\
\frac43
\end{array}
\right](z,\tau)      \notag \\
&
\hspace{30mm}
-
\theta^2
\left[
\begin{array}{c}
1 \\
\frac13
\end{array}
\right]
\theta^2
\left[
\begin{array}{c}
1 \\
1
\end{array}
\right](z,\tau)=0.  \label{eqn:two-theta-sum-2-deno-3-(0)}
\end{align}
}
\end{proposition}

\begin{proof}
We prove equation (\ref{eqn:two-theta-sum-2-deno-3-(0)}). 
We first note that 
$
\dim \mathcal{F}_{2}\left[
\begin{array}{c}
0 \\
0
\end{array}
\right]
=
2,
$ 
and 
\begin{equation*}
\theta^2
\left[
\begin{array}{c}
1 \\
0
\end{array}
\right](z,\tau), \,
\theta 
\left[
\begin{array}{c}
1 \\
\frac23
\end{array}
\right](z,\tau)
\theta
\left[
\begin{array}{c}
1 \\
\frac43
\end{array}
\right](z,\tau), \,
\theta^2
\left[
\begin{array}{c}
1 \\
1
\end{array}
\right](z,\tau) \in
\mathcal{F}_{2}\left[
\begin{array}{c}
0 \\
0
\end{array}
\right].  
\end{equation*}
Therefore, there exist some complex numbers, $x_1, x_2$ and $ x_3$ not all zero such that 
 \begin{align*}
 &
 x_1
\theta^2
\left[
\begin{array}{c}
1 \\
0
\end{array}
\right](z,\tau)
+x_2
 \theta
\left[
\begin{array}{c}
1 \\
\frac23
\end{array}
\right](z,\tau)
\theta
\left[
\begin{array}{c}
1 \\
\frac43
\end{array}
\right](z,\tau)
+
x_3
\theta^2
\left[
\begin{array}{c}
1 \\
1
\end{array}
\right](z,\tau)=0. 
\end{align*}
Note that in the fundamental parallelogram, 
the zero of 
$
\theta
\left[
\begin{array}{c}
1 \\
0
\end{array}
\right](z),
$ 
$
\theta
\left[
\begin{array}{c}
1 \\
\frac23
\end{array}
\right](z),
$ 
or 
$
\theta
\left[
\begin{array}{c}
1 \\
1
\end{array}
\right](z)
$ 
is $z=1/2,$ $1/6$ or $0.$ 
Substituting $z=1/2, 1/6,$ and $0,$ we have 
\begin{alignat*}{4}
&
&    
&
x_2
\theta^2\left[
\begin{array}{c}
1 \\
\frac13
\end{array}
\right]  
&    
&+x_3\theta^2
\left[
\begin{array}{c}
1 \\
0
\end{array}
\right]  
&  
&=0,  \\
&x_1
\theta^2
\left[
\begin{array}{c}
1 \\
\frac13
\end{array}
\right]  
&  
&
&
&+x_3
\theta^2
\left[
\begin{array}{c}
1 \\
\frac23
\end{array}
\right]  
&  
&=0,  \\
&x_1\theta^2\left[
\begin{array}{c}
1 \\
0
\end{array}
\right]  
&
-&x_2\theta^2\left[
\begin{array}{c}
1 \\
\frac23
\end{array}
\right] 
&
&
&
&=0. 
\end{alignat*}
Solving this system of equations, 
we have 
\begin{equation*}
(x_1,x_2,x_3)=\alpha
\left(
\theta^2\left[
\begin{array}{c}
1 \\
\frac23
\end{array}
\right], 
\theta^2\left[
\begin{array}{c}
1 \\
0
\end{array}
\right], 
-
\theta^2\left[
\begin{array}{c}
1 \\
\frac13
\end{array}
\right] 
\right) \,\,\,   \text{for some} \,\,\alpha\in\mathbb{C}\setminus\{0\},
\end{equation*}
which proves the proposition. 
\end{proof}

\begin{theorem}
\label{thm:1-2/3}
{\it
For every $\tau\in\mathbb{H}^2,$ we have 
\begin{equation*}
\frac{d}{d\tau} 
\log \frac{\eta^4(6\tau)}{\eta^3(\tau) \eta(3\tau)}
+
\frac{1}{2\pi i } 
\left\{
\frac
{
\theta^{\prime} 
\left[
\begin{array}{c}
1 \\
\frac23
\end{array}
\right](0,\tau) 
}
{
\theta
\left[
\begin{array}{c}
1 \\
\frac23
\end{array}
\right](0,\tau)
}
\right\}^2
=0. 
\end{equation*}
}
\end{theorem}

\begin{proof}
Comparing the coefficients of the term $z^2$ in equation (\ref{eqn:two-theta-sum-2-deno-3-(0)}) yields 
\begin{equation*}
\frac
{
\theta^{\prime \prime} 
\left[
\begin{array}{c}
1 \\
0
\end{array}
\right]
}
{
\theta
\left[
\begin{array}{c}
1 \\
0
\end{array}
\right]
}
-
\frac
{
\theta^{\prime \prime} 
\left[
\begin{array}{c}
1 \\
\frac23
\end{array}
\right]
}
{
\theta
\left[
\begin{array}{c}
1 \\
\frac23
\end{array}
\right]
}
-
\frac
{
\left\{
\theta^{\prime } 
\left[
\begin{array}{c}
1 \\
1
\end{array}
\right]
\right\}^2
\theta^2
\left[
\begin{array}{c}
1 \\
\frac13
\end{array}
\right]
}
{
\theta^2
\left[
\begin{array}{c}
1 \\
0
\end{array}
\right]
\theta^2
\left[
\begin{array}{c}
1 \\
\frac23
\end{array}
\right]
}
+
\left\{
\frac
{
\theta^{\prime} 
\left[
\begin{array}{c}
1 \\
\frac23
\end{array}
\right]
}
{
\theta
\left[
\begin{array}{c}
1 \\
\frac23
\end{array}
\right]
}
\right\}^2
=0. 
\end{equation*}
\par
The heat equation (\ref{eqn:heat}) and equation (\ref{eqn:relation-1133-(2)}) imply that 
\begin{equation*}
2\pi i \frac{d}{d\tau}
\log 
\left\{
\frac
{
\theta^2
\left[
\begin{array}{c}
1 \\
0
\end{array}
\right]
\theta^3
\left[
\begin{array}{c}
1 \\
\frac13
\end{array}
\right]
\theta^4
\left[
\begin{array}{c}
1 \\
\frac23
\end{array}
\right]
}
{
\theta^{\prime}
\left[
\begin{array}{c}
1 \\
1
\end{array}
\right]^3
}
\right\}
+
\left\{
\frac
{
\theta^{\prime} 
\left[
\begin{array}{c}
1 \\
\frac23
\end{array}
\right]
}
{
\theta
\left[
\begin{array}{c}
1 \\
\frac23
\end{array}
\right]
}
\right\}^2
=0. 
\end{equation*}
Jacobi's triple product identity (\ref{eqn:Jacobi-triple}) yields 
\begin{equation*}
\frac
{
\theta^2
\left[
\begin{array}{c}
1 \\
0
\end{array}
\right]
\theta^3
\left[
\begin{array}{c}
1 \\
\frac13
\end{array}
\right]
\theta^4
\left[
\begin{array}{c}
1 \\
\frac23
\end{array}
\right]
}
{
\theta^{\prime}
\left[
\begin{array}{c}
1 \\
1
\end{array}
\right]^3
}
=
-
\frac{3\sqrt{3}}{2\pi^3}
\frac
{
\eta^4(6\tau)
}
{
\eta^3(\tau) \eta(3\tau)
},
\end{equation*}
which proves the theorem. 
\end{proof}

\subsubsection{On $x^2+xy+y^2+2(z^2+zw+w^2)$}

\begin{theorem}
\label{thm:x^2+xy+y^2-(1,2)}
{\it
For each $n\in\mathbb{N},$ set 
\begin{equation*}
s_{1,2}(n)
=
\sharp
\left\{
(x,y,z,w)\in\mathbb{Z}^4 \, | \, 
x^2+xy+y^2+2(z^2+zw+w^2)=n
\right\}. 
\end{equation*}
Then, we have 
\begin{equation*}
s_{1,2}(n)=
6\sigma(n)-12\sigma(n/2)+18\sigma(n/3)-36\sigma(n/6). 
\end{equation*}
}
\end{theorem}

\begin{proof}
Set $q=\exp(2\pi i \tau).$ 
Summing both sides of equations (\ref{eqn:relation-for-thm-1-1/3}) and (\ref{eqn:relation-1133-(1)}) yields 
\begin{equation*}
2
\frac
{
\theta^{\prime \prime} 
\left[
\begin{array}{c}
1 \\
\frac13
\end{array}
\right]
}
{
\theta
\left[
\begin{array}{c}
1 \\
\frac13
\end{array}
\right]
}
+
\frac
{
\theta^{\prime \prime} 
\left[
\begin{array}{c}
1 \\
\frac23
\end{array}
\right]
}
{
\theta
\left[
\begin{array}{c}
1 \\
\frac23
\end{array}
\right]
}
-
\frac
{
\theta^{\prime \prime \prime} 
\left[
\begin{array}{c}
1 \\
1
\end{array}
\right]
}
{
\theta^{\prime}
\left[
\begin{array}{c}
1 \\
1
\end{array}
\right]
}
+
2
\frac
{
\theta^{\prime} 
\left[
\begin{array}{c}
1 \\
\frac13
\end{array}
\right]
}
{
\theta
\left[
\begin{array}{c}
1 \\
\frac13
\end{array}
\right]
}
\left\{
\frac
{
\theta^{\prime} 
\left[
\begin{array}{c}
1 \\
\frac13
\end{array}
\right]
}
{
\theta
\left[
\begin{array}{c}
1 \\
\frac13
\end{array}
\right]
}
-
\frac
{
\theta^{\prime} 
\left[
\begin{array}{c}
1 \\
\frac23
\end{array}
\right]
}
{
\theta
\left[
\begin{array}{c}
1 \\
\frac23
\end{array}
\right]
}
\right\}
+
\left\{
\frac
{
\theta^{\prime} 
\left[
\begin{array}{c}
1 \\
\frac13
\end{array}
\right]
}
{
\theta
\left[
\begin{array}{c}
1 \\
\frac13
\end{array}
\right]
}
\right\}^2
+
\left\{
\frac
{
\theta^{\prime} 
\left[
\begin{array}{c}
1 \\
\frac23
\end{array}
\right]
}
{
\theta
\left[
\begin{array}{c}
1 \\
\frac23
\end{array}
\right]
}
\right\}^2
=0. 
\end{equation*}
The heat equation (\ref{eqn:heat}), Proposition \ref{prop:preliminary-1/3-2/3} and Theorems \ref{thm:1-1/3} and \ref{thm:1-2/3} imply 
that 
\begin{equation}
\label{eqn:product-1-1/3-tau-2tau}
\frac
{
\theta^{\prime} 
\left[
\begin{array}{c}
1 \\
\frac13
\end{array}
\right](0,\tau)
}
{
\theta
\left[
\begin{array}{c}
1 \\
\frac13
\end{array}
\right](0,\tau)
}
\cdot
\frac
{
\theta^{\prime} 
\left[
\begin{array}{c}
1 \\
\frac13
\end{array}
\right](0,2\tau)
}
{
\theta
\left[
\begin{array}{c}
1 \\
\frac13
\end{array}
\right](0,2\tau)
}
=
\frac{2\pi i}{2}
\frac{d}{d\tau}
\log 
\left\{
\frac
{
\eta(\tau)\eta(3\tau)
}
{
\eta(2\tau)\eta(6\tau)
}
\right\}. 
\end{equation}
Proposition \ref{prop:preliminary-1/3-2/3} shows that 
\begin{align*}
a(q) a(q^2)=&
6q \frac{d}{dq}  \log q^{\frac16} \prod_{n=1}^{\infty} \frac{(1-q^{2n}) (1-q^{6n})}{(1-q^{n}) (1-q^{3n})}  \\
=&
1+6\sum_{n=1}^{\infty} (\sigma(n)-2\sigma(n/2)+3\sigma(n/3)-6\sigma(n/6)) q^n. 
\end{align*}
\end{proof}

\subsection{Theorem for $\epsilon=0, \epsilon^{\prime}=1/3$}

\begin{theorem}
\label{thm:0-1/3}
{\it
For every $\tau\in\mathbb{H}^2,$ we have 
\begin{equation*}
\frac{d}{d\tau} 
\log \frac{\eta^4(3\tau/2)}{\eta^3(\tau) \eta(3\tau)}
+
\frac{1}{2\pi i } 
\left\{
\frac
{
\theta^{\prime} 
\left[
\begin{array}{c}
0 \\
\frac13
\end{array}
\right](0,\tau) 
}
{
\theta
\left[
\begin{array}{c}
0 \\
\frac13
\end{array}
\right](0,\tau)
}
\right\}^2
=0. 
\end{equation*}
}
\end{theorem}

\begin{proof}
Proposition \ref{prop:preliminary-1/3-2/3} implies that 
\begin{align*}
&
\left\{
\frac
{
\theta^{\prime} 
\left[
\begin{array}{c}
0 \\
\frac13
\end{array}
\right](0,\tau) 
}
{
\theta
\left[
\begin{array}{c}
0 \\
\frac13
\end{array}
\right](0,\tau)
}
\right\}^2  \\
=&
\left\{
\frac
{
\theta^{\prime} 
\left[
\begin{array}{c}
1 \\
\frac13
\end{array}
\right](0,\tau/2) 
}
{
\theta
\left[
\begin{array}{c}
1 \\
\frac13
\end{array}
\right](0,\tau/2)
}
\right\}^2
-2
\frac
{
\theta^{\prime} 
\left[
\begin{array}{c}
1 \\
\frac13
\end{array}
\right](0,\tau/2) 
}
{
\theta
\left[
\begin{array}{c}
1 \\
\frac13
\end{array}
\right](0,\tau/2)
}
\cdot
\frac
{
\theta^{\prime} 
\left[
\begin{array}{c}
1 \\
\frac13
\end{array}
\right](0,\tau) 
}
{
\theta
\left[
\begin{array}{c}
1 \\
\frac13
\end{array}
\right](0,\tau)
}
+
\left\{
\frac
{
\theta^{\prime} 
\left[
\begin{array}{c}
1 \\
\frac13
\end{array}
\right](0,\tau) 
}
{
\theta
\left[
\begin{array}{c}
1 \\
\frac13
\end{array}
\right](0,\tau)
}
\right\}^2. 
\end{align*}
Theorem \ref{thm:0-1/3} and equation (\ref{eqn:product-1-1/3-tau-2tau}) imply that 
\begin{align*}
&
\left\{
\frac
{
\theta^{\prime} 
\left[
\begin{array}{c}
0 \\
\frac13
\end{array}
\right](0,\tau) 
}
{
\theta
\left[
\begin{array}{c}
0 \\
\frac13
\end{array}
\right](0,\tau)
}
\right\}^2  \\
=&
-2\cdot 2\pi i
\frac{d}{d\tau} \log \frac{\eta(3\tau/2)}{\eta(\tau/2)} 
-2\cdot 
2 \cdot \frac{2\pi i}{2}
\frac{d}{d\tau} \log \frac{\eta(\tau/2) \eta(3\tau/2)}{\eta(\tau) \eta(3\tau)} 
-2\pi i 
\frac{d}{d\tau} \log \frac{\eta(3\tau)}{\eta(\tau)},
\end{align*}
which proves the theorem. 
\end{proof}

\subsection{Theorem for $\epsilon=0, \epsilon^{\prime}=2/3$}

\begin{theorem}
\label{thm:0-2/3}
{\it
For every $\tau\in\mathbb{H}^2,$ we have 
\begin{equation*}
\frac{d}{d\tau} 
\log \frac{\eta^{11}(3\tau)}{\eta^3(\tau) \eta^4(3\tau/2) \eta^4(6\tau)}
+
\frac{1}{2\pi i } 
\left\{
\frac
{
\theta^{\prime} 
\left[
\begin{array}{c}
0 \\
\frac23
\end{array}
\right](0,\tau) 
}
{
\theta
\left[
\begin{array}{c}
0 \\
\frac23
\end{array}
\right](0,\tau)
}
\right\}^2
=0. 
\end{equation*}
}
\end{theorem}

\begin{proof}
Consider the following elliptic functions:
\begin{equation*}
\varphi(z)
=
\frac
{
\theta^2
\left[
\begin{array}{c}
0 \\
\frac13
\end{array}
\right](z)
\theta
\left[
\begin{array}{c}
1 \\
\frac13
\end{array}
\right](z) 
}
{
\theta^3
\left[
\begin{array}{c}
1 \\
1
\end{array}
\right](z)
}, \,\,
\psi(z)
=
\frac
{
\theta^2
\left[
\begin{array}{c}
0 \\
\frac23
\end{array}
\right](z)
\theta
\left[
\begin{array}{c}
1 \\
-\frac13
\end{array}
\right](z) 
}
{
\theta^3
\left[
\begin{array}{c}
1 \\
1
\end{array}
\right](z)
}. 
\end{equation*}
In the fundamental parallelogram, 
the pole of $\varphi(z)$ or $\psi(z)$ is $z=0,$ 
which implies that 
$\mathrm{Res} (\varphi(z), 0)=\mathrm{Res} (\psi(z), 0)=0.$ 
Therefore, it follows that 
\begin{equation*}
\frac
{
\theta^{\prime\prime}
\left[
\begin{array}{c}
1 \\
\frac13
\end{array}
\right]
}
{
\theta
\left[
\begin{array}{c}
1 \\
\frac13
\end{array}
\right]
}
+
2
\frac
{
\theta^{\prime\prime}
\left[
\begin{array}{c}
0 \\
\frac13
\end{array}
\right]
}
{
\theta
\left[
\begin{array}{c}
0 \\
\frac13
\end{array}
\right]
}
-
\frac
{
\theta^{\prime\prime \prime}
\left[
\begin{array}{c}
1 \\
1
\end{array}
\right]
}
{
\theta^{\prime}
\left[
\begin{array}{c}
1 \\
1
\end{array}
\right]
}
+
4
\frac
{
\theta^{\prime} 
\left[
\begin{array}{c}
0 \\
\frac13
\end{array}
\right]
}
{
\theta
\left[
\begin{array}{c}
0 \\
\frac13
\end{array}
\right]
}
\cdot
\frac
{
\theta^{\prime} 
\left[
\begin{array}{c}
1 \\
\frac13
\end{array}
\right]
}
{
\theta
\left[
\begin{array}{c}
1 \\
\frac13
\end{array}
\right]
}
+
2
\left\{
\frac
{
\theta^{\prime} 
\left[
\begin{array}{c}
0 \\
\frac13
\end{array}
\right]
}
{
\theta
\left[
\begin{array}{c}
0 \\
\frac13
\end{array}
\right]
}
\right\}^2=0,
\end{equation*}
and 
\begin{equation*}
\frac
{
\theta^{\prime\prime}
\left[
\begin{array}{c}
1 \\
\frac13
\end{array}
\right]
}
{
\theta
\left[
\begin{array}{c}
1 \\
\frac13
\end{array}
\right]
}
+
2
\frac
{
\theta^{\prime\prime}
\left[
\begin{array}{c}
0 \\
\frac23
\end{array}
\right]
}
{
\theta
\left[
\begin{array}{c}
0 \\
\frac23
\end{array}
\right]
}
-
\frac
{
\theta^{\prime\prime \prime}
\left[
\begin{array}{c}
1 \\
1
\end{array}
\right]
}
{
\theta^{\prime}
\left[
\begin{array}{c}
1 \\
1
\end{array}
\right]
}
-
4
\frac
{
\theta^{\prime} 
\left[
\begin{array}{c}
0 \\
\frac23
\end{array}
\right]
}
{
\theta
\left[
\begin{array}{c}
0 \\
\frac23
\end{array}
\right]
}
\cdot
\frac
{
\theta^{\prime} 
\left[
\begin{array}{c}
1 \\
\frac13
\end{array}
\right]
}
{
\theta
\left[
\begin{array}{c}
1 \\
\frac13
\end{array}
\right]
}
+
2
\left\{
\frac
{
\theta^{\prime} 
\left[
\begin{array}{c}
0 \\
\frac23
\end{array}
\right]
}
{
\theta
\left[
\begin{array}{c}
0 \\
\frac23
\end{array}
\right]
}
\right\}^2=0.
\end{equation*}
Summing both sides of these equations yields 
\begin{align*}
&
\frac
{
\theta^{\prime\prime}
\left[
\begin{array}{c}
1 \\
\frac13
\end{array}
\right]
}
{
\theta
\left[
\begin{array}{c}
1 \\
\frac13
\end{array}
\right]
}
+
\frac
{
\theta^{\prime\prime}
\left[
\begin{array}{c}
0 \\
\frac13
\end{array}
\right]
}
{
\theta
\left[
\begin{array}{c}
0 \\
\frac13
\end{array}
\right]
}
+
\frac
{
\theta^{\prime\prime}
\left[
\begin{array}{c}
0 \\
\frac23
\end{array}
\right]
}
{
\theta
\left[
\begin{array}{c}
0 \\
\frac23
\end{array}
\right]
}
-
\frac
{
\theta^{\prime\prime \prime}
\left[
\begin{array}{c}
1 \\
1
\end{array}
\right]
}
{
\theta^{\prime}
\left[
\begin{array}{c}
1 \\
1
\end{array}
\right]
}  \\
&
+
2\left\{
\frac
{
\theta^{\prime}
\left[
\begin{array}{c}
0 \\
\frac13
\end{array}
\right]
}
{
\theta
\left[
\begin{array}{c}
0 \\
\frac13
\end{array}
\right]
}
-
\frac
{
\theta^{\prime}
\left[
\begin{array}{c}
0 \\
\frac23
\end{array}
\right]
}
{
\theta
\left[
\begin{array}{c}
0 \\
\frac23
\end{array}
\right]
}
\right\}
\frac
{
\theta^{\prime}
\left[
\begin{array}{c}
1 \\
\frac13
\end{array}
\right]
}
{
\theta
\left[
\begin{array}{c}
1\\
\frac13
\end{array}
\right]
}
+
\left\{
\frac
{
\theta^{\prime}
\left[
\begin{array}{c}
0 \\
\frac13
\end{array}
\right]
}
{
\theta
\left[
\begin{array}{c}
0 \\
\frac13
\end{array}
\right]
}
\right\}^2
+
\left\{
\frac
{
\theta^{\prime}
\left[
\begin{array}{c}
0 \\
\frac23
\end{array}
\right]
}
{
\theta
\left[
\begin{array}{c}
0 \\
\frac23
\end{array}
\right]
}
\right\}^2=0. 
\end{align*}
The theorem follows from 
the heat equation (\ref{eqn:heat}), Proposition \ref{prop:preliminary-1/3-2/3}, equation (\ref{eqn:product-1-1/3-tau-2tau}) and Theorems \ref{thm:1-1/3}, \ref{thm:0-1/3}. 
\end{proof}

\subsection{Convolution sums}

\begin{theorem}
\label{thm:convolution-delta-epsilon}
{\it
For each $n\in\mathbb{N},$ set 
\begin{equation*}
\delta^{*}(n)=d^{*}_{1,3}(n)-d^{*}_{2,3}(n), \,\,
\text{and} \,\,
\epsilon^{*}(n)=d^{*}_{1,6}(n)+d^{*}_{2,6}(n)-d^{*}_{4,6}(n)-d^{*}_{5,6}(n).
\end{equation*}
Then, for $n\ge 2,$ 
\begin{equation}
\label{eqn:delta}
\sum_{k=1}^{n-1}
\delta^{*}(k)\delta^{*}(n-k)
=
\sigma(n/2)-2\sigma(n/3)+\sigma(n/6), 
\end{equation}
and
\begin{equation}
\label{eqn:epsilon}
\sum_{k=1}^{n-1}
\epsilon^{*}(k)\epsilon^{*}(n-k)
=
\sigma(n/2)+2\sigma(n/3)-11\sigma(n/6)+8\sigma(n/12). 
\end{equation}
}
\end{theorem}

\begin{proof}
For every $\tau\in\mathbb{H}^2,$ set $x=\exp(\pi i \tau).$
Jacobi's triple product identity (\ref{eqn:Jacobi-triple}) yields 
\begin{equation*}
\frac
{
\theta^{\prime} 
\left[
\begin{array}{c}
0 \\
\frac13
\end{array}
\right]
}
{
\theta
\left[
\begin{array}{c}
0 \\
\frac13
\end{array}
\right]
}
=
-2\sqrt{3}\pi \sum_{n=1}^{\infty}(d^{*}_{1,3}(n)-d_{2,3}^{*}(n)) x^n,
\end{equation*}
and
\begin{equation*}
\frac
{
\theta^{\prime} 
\left[
\begin{array}{c}
0 \\
\frac23
\end{array}
\right]
}
{
\theta
\left[
\begin{array}{c}
0 \\
\frac23
\end{array}
\right]
}
=
-2\sqrt{3}\pi \sum_{n=1}^{\infty}(d^{*}_{1,6}(n)+d_{2,6}^{*}(n)-d_{4,6}^{*}(n)-d^{*}_{5,6}(n)) x^n. 
\end{equation*}
The theorem follows from Theorems \ref{thm:0-1/3} and \ref{thm:0-2/3}. 
\end{proof}

\begin{theorem}
\label{thm:convolution-delta*epsilon}
{\it
For each $n\in\mathbb{N},$ set 
\begin{equation*}
\delta^{*}(n)=d^{*}_{1,3}(n)-d^{*}_{2,3}(n), \,\,
\text{and} \,\,
\epsilon^{*}(n)=d^{*}_{1,6}(n)+d^{*}_{2,6}(n)-d^{*}_{4,6}(n)-d^{*}_{5,6}(n).
\end{equation*}
Then, for $n\ge 2,$ 
\begin{equation*}
\label{eqn:delta}
\sum_{k=1}^{n-1}
\delta^{*}(k)\epsilon^{*}(n-k)
=
\begin{cases}
0 &\text{if $n$ is odd,} \\
\sigma(n)-2\sigma(n/2)-\sigma(n/3)+12\sigma(n/6)  &\text{if $n$ is even.}
\end{cases}
\end{equation*}
}
\end{theorem}

\begin{proof}
Set $x=\exp(\pi i \tau)$ and $q=\exp(2\pi i \tau).$ 
Consider the following elliptic function:
\begin{equation*}
\varphi(z)
=
\frac
{
\theta
\left[
\begin{array}{c}
0 \\
\frac13
\end{array}
\right](z)
\theta
\left[
\begin{array}{c}
0 \\
\frac23
\end{array}
\right](z) 
\theta
\left[
\begin{array}{c}
1 \\
0
\end{array}
\right](z) 
}
{
\theta^3
\left[
\begin{array}{c}
1 \\
1
\end{array}
\right](z)
}. 
\end{equation*}
In the fundamental parallelogram, 
the pole of $\varphi(z)$ is $z=0,$ 
which implies that 
$\mathrm{Res} (\varphi(z), 0)=0.$ 
Therefore, it follows that 
\begin{equation*}
\frac12
\frac
{
\theta^{\prime\prime}
\left[
\begin{array}{c}
0 \\
\frac13
\end{array}
\right]
}
{
\theta
\left[
\begin{array}{c}
0 \\
\frac13
\end{array}
\right]
}
+
\frac12
\frac
{
\theta^{\prime\prime}
\left[
\begin{array}{c}
0 \\
\frac23
\end{array}
\right]
}
{
\theta
\left[
\begin{array}{c}
0 \\
\frac23
\end{array}
\right]
}
+
\frac12
\frac
{
\theta^{\prime\prime}
\left[
\begin{array}{c}
1 \\
0
\end{array}
\right]
}
{
\theta
\left[
\begin{array}{c}
1 \\
0
\end{array}
\right]
}
-
\frac12
\frac
{
\theta^{\prime\prime \prime}
\left[
\begin{array}{c}
1 \\
1
\end{array}
\right]
}
{
\theta^{\prime}
\left[
\begin{array}{c}
1 \\
1
\end{array}
\right]
}
+
\frac
{
\theta^{\prime} 
\left[
\begin{array}{c}
0 \\
\frac13
\end{array}
\right]
}
{
\theta
\left[
\begin{array}{c}
0 \\
\frac13
\end{array}
\right]
}
\cdot
\frac
{
\theta^{\prime} 
\left[
\begin{array}{c}
0 \\
\frac23
\end{array}
\right]
}
{
\theta
\left[
\begin{array}{c}
0 \\
\frac23
\end{array}
\right]
}
=0. 
\end{equation*}
The heat equation (\ref{eqn:heat}) and Jacobi's triple product identity (\ref{eqn:Jacobi-triple}) yield
\begin{equation*}
\frac
{
\theta^{\prime} 
\left[
\begin{array}{c}
0 \\
\frac13
\end{array}
\right]
}
{
\theta
\left[
\begin{array}{c}
0 \\
\frac13
\end{array}
\right]
}
\cdot
\frac
{
\theta^{\prime} 
\left[
\begin{array}{c}
0 \\
\frac23
\end{array}
\right]
}
{
\theta
\left[
\begin{array}{c}
0 \\
\frac23
\end{array}
\right]
}
=
-2\pi i \frac{d}{d\tau}
\log
\frac{\eta(3\tau)\eta^3(2\tau)}{\eta^3(\tau) \eta(6\tau)},
\end{equation*}
which proves the theorem. 
\end{proof}

\begin{theorem}
\label{thm:convolution-delta-epsilon-k+2l}
{\it
For each $n\in\mathbb{N},$ set 
\begin{equation*}
\delta^{*}(n)=d^{*}_{1,3}(n)-d^{*}_{2,3}(n), \,\,
\text{and} \,\,
\epsilon^{*}(n)=d^{*}_{1,6}(n)+d^{*}_{2,6}(n)-d^{*}_{4,6}(n)-d^{*}_{5,6}(n).
\end{equation*}
Then, for $n\ge 3,$ 
\begin{equation*}
\label{eqn:delta}
\sum_{
\tiny{
\begin{matrix}
(k,l)\in\mathbb{N}^2 \\ 
2k+l=n
\end{matrix}
}
}
\delta^{*}(k)\delta^{*}(l)
=
\sigma(n/3)-\sigma(n/4)-\sigma(n/6)+\sigma(n/12), 
\end{equation*}
\begin{equation*}
\label{eqn:delta}
\sum_{
\tiny{
\begin{matrix}
(k,l)\in\mathbb{N}^2 \\ 
2k+l=n
\end{matrix}
}
}
\delta^{*}(k)\epsilon^{*}(l)
=
\sigma(n/3)+\sigma(n/4)-5\sigma(n/6)+3\sigma(n/12).
\end{equation*}
}
\end{theorem}

\begin{proof}
For every $\tau\in\mathbb{H}^2,$ 
set $x=\exp(\pi i \tau).$ 
Note that 
\begin{equation*}
\frac
{
\theta^{\prime} 
\left[
\begin{array}{c}
0 \\
\frac13
\end{array}
\right]
}
{
\theta
\left[
\begin{array}{c}
0 \\
\frac13
\end{array}
\right]
}
-
\frac
{
\theta^{\prime} 
\left[
\begin{array}{c}
0 \\
\frac23
\end{array}
\right]
}
{
\theta
\left[
\begin{array}{c}
0 \\
\frac23
\end{array}
\right]
}
=
4\sqrt{3}\pi \sum_{n=1}^{\infty} (d_{1,3}^{*}(n)-d_{2,3}^{*}(n)) x^{2n}. 
\end{equation*}
The theorem can be proved by considering 
\begin{equation*}
\left\{
\frac
{
\theta^{\prime} 
\left[
\begin{array}{c}
0 \\
\frac13
\end{array}
\right]
}
{
\theta
\left[
\begin{array}{c}
0 \\
\frac13
\end{array}
\right]
}
-
\frac
{
\theta^{\prime} 
\left[
\begin{array}{c}
0 \\
\frac23
\end{array}
\right]
}
{
\theta
\left[
\begin{array}{c}
0 \\
\frac23
\end{array}
\right]
}
\right\}
\frac
{
\theta^{\prime} 
\left[
\begin{array}{c}
0 \\
\frac13
\end{array}
\right]
}
{
\theta
\left[
\begin{array}{c}
0 \\
\frac13
\end{array}
\right]
}, \,\,
\mathrm{and} \,\,
\left\{
\frac
{
\theta^{\prime} 
\left[
\begin{array}{c}
0 \\
\frac13
\end{array}
\right]
}
{
\theta
\left[
\begin{array}{c}
0 \\
\frac13
\end{array}
\right]
}
-
\frac
{
\theta^{\prime} 
\left[
\begin{array}{c}
0 \\
\frac23
\end{array}
\right]
}
{
\theta
\left[
\begin{array}{c}
0 \\
\frac23
\end{array}
\right]
}
\right\}
\frac
{
\theta^{\prime} 
\left[
\begin{array}{c}
0 \\
\frac23
\end{array}
\right]
}
{
\theta
\left[
\begin{array}{c}
0 \\
\frac23
\end{array}
\right]
}. 
\end{equation*}
\end{proof}

\subsubsection*{Remark}
For each $n\in\mathbb{N},$ set $\delta(n)=d_{1,3}(n)-d_{2,3}(n).$ 
Using the theory of theta functions, 
Farkas \cite{Farkas} showed that for each $n\in\mathbb{N}_0,$ 
\begin{equation*}
\sigma(3n+2)
=
3\sum_{k=0}^n 
\delta(3k+1) \delta(3(n-k)+1).
\end{equation*}

\section{Case for $\Gamma_{o}(8)$}

\subsection{Theorem for $\epsilon=1, \epsilon^{\prime}=1/4, 3/4$}

\begin{theorem}
\label{thm:1-1/4,3/4}
{\it
For every $\tau\in\mathbb{H}^2,$ 
we have 
\begin{equation*}
\frac{d}{d\tau} 
\log \frac{\eta^3(8\tau)}{\eta^2(\tau) \eta(4\tau)}
+
\frac{1}{2\pi i \cdot 2} 
\left[
\left\{
\frac
{
\theta^{\prime} 
\left[
\begin{array}{c}
1 \\
\frac14
\end{array}
\right](0,\tau) 
}
{
\theta
\left[
\begin{array}{c}
1 \\
\frac14
\end{array}
\right](0,\tau)
}
\right\}^2
+
\left\{
\frac
{
\theta^{\prime} 
\left[
\begin{array}{c}
1 \\
\frac34
\end{array}
\right](0,\tau) 
}
{
\theta
\left[
\begin{array}{c}
1 \\
\frac34
\end{array}
\right](0,\tau)
}
\right\}^2
\right]=0. 
\end{equation*}
}
\end{theorem}

\begin{proof}
By Lemma \ref{lem:Farkas-Kra}, we first note that 
\begin{equation*}
\theta^2
\left[
\begin{array}{c}
0 \\
0
\end{array}
\right](0,\tau)
=
\theta^2
\left[
\begin{array}{c}
0 \\
0
\end{array}
\right](0,2\tau)
+
\theta^2
\left[
\begin{array}{c}
1\\
0
\end{array}
\right](0,2\tau),
\end{equation*}
and 
\begin{equation*}
\theta^2
\left[
\begin{array}{c}
1 \\
0
\end{array}
\right](0,\tau)
=
2
\theta
\left[
\begin{array}{c}
0 \\
0
\end{array}
\right](0,2\tau)
\theta
\left[
\begin{array}{c}
1\\
0
\end{array}
\right](0,2\tau),
\end{equation*}
which implies that 
\begin{equation*}
\theta^4
\left[
\begin{array}{c}
1 \\
0
\end{array}
\right](0,\tau)
=
4
\theta^2
\left[
\begin{array}{c}
0 \\
0
\end{array}
\right](0,2\tau)
\theta^2
\left[
\begin{array}{c}
1\\
0
\end{array}
\right](0,2\tau).
\end{equation*}
The derivative formulas (\ref{eqn:analogue-1-1/4}) and (\ref{eqn:analogue-1-3/4}) yield 
\begin{equation*}
\left\{
\frac
{
\theta^{\prime} 
\left[
\begin{array}{c}
1 \\
\frac14
\end{array}
\right](0,\tau) 
}
{
\theta
\left[
\begin{array}{c}
1 \\
\frac14
\end{array}
\right](0,\tau)
}
\right\}^2
+
\left\{
\frac
{
\theta^{\prime} 
\left[
\begin{array}{c}
1 \\
\frac34
\end{array}
\right](0,\tau) 
}
{
\theta
\left[
\begin{array}{c}
1 \\
\frac34
\end{array}
\right](0,\tau)
}
\right\}^2
=
6\pi^2
\theta^4
\left[
\begin{array}{c}
0 \\
0
\end{array}
\right](0,4\tau)
+
\pi^2
\theta^4
\left[
\begin{array}{c}
1 \\
0
\end{array}
\right](0,2\tau). 
\end{equation*}
By the proof of Theorems \ref{thm:4-squares} and \ref{thm:4-triangular}, we have 
\begin{align*}
6\pi^2
\theta^4
\left[
\begin{array}{c}
0 \\
0
\end{array}
\right](0,4\tau)
+
\pi^2
\theta^4
\left[
\begin{array}{c}
1 \\
0
\end{array}
\right](0,\tau)
=&
6\cdot 4\pi i
\cdot \left
(-\frac12\right)
\cdot 
\frac{d}{d\tau}
\log \frac{\eta(8\tau)}{\eta(2\tau)}
-
4\pi i
\frac{d}{d\tau}
\log
\frac{\eta^3(2\tau)}{\eta^2(\tau)\eta(4\tau)}  \\
=&
4\pi i 
\frac{d}{d\tau}
\log
\frac{\eta^2(\tau)\eta(4\tau)}{\eta^3(8\tau)}, 
\end{align*}
which proves the theorem. 
\end{proof}

\subsection{Theorem for $\epsilon=0, \epsilon^{\prime}=1/4, 3/4$}

\begin{theorem}
\label{thm:0-1/4,3/4}
{\it
For every $\tau\in\mathbb{H}^2,$ 
we have 
\begin{equation*}
\frac{d}{d\tau} 
\log \frac{\eta^8(4\tau)}{\eta^2(\tau) \eta^3(2\tau) \eta^3(8\tau)}
+
\frac{1}{2\pi i \cdot 2} 
\left[
\left\{
\frac
{
\theta^{\prime} 
\left[
\begin{array}{c}
0 \\
\frac14
\end{array}
\right](0,\tau) 
}
{
\theta
\left[
\begin{array}{c}
0 \\
\frac14
\end{array}
\right](0,\tau)
}
\right\}^2
+
\left\{
\frac
{
\theta^{\prime} 
\left[
\begin{array}{c}
0 \\
\frac34
\end{array}
\right](0,\tau) 
}
{
\theta
\left[
\begin{array}{c}
0 \\
\frac34
\end{array}
\right](0,\tau)
}
\right\}^2
\right]=0. 
\end{equation*}
}
\end{theorem}

\begin{proof}
By Lemma \ref{lem:Farkas-Kra}, we first note that
\begin{equation*}
\theta^2
\left[
\begin{array}{c}
0 \\
0
\end{array}
\right](0,\tau)
=
\theta^2
\left[
\begin{array}{c}
0 \\
0
\end{array}
\right](0,2\tau)
+
\theta^2
\left[
\begin{array}{c}
1\\
0
\end{array}
\right](0,2\tau),
\end{equation*}
and 
\begin{equation*}
\theta^2
\left[
\begin{array}{c}
1 \\
0
\end{array}
\right](0,\tau)
=
2
\theta
\left[
\begin{array}{c}
0 \\
0
\end{array}
\right](0,2\tau)
\theta
\left[
\begin{array}{c}
1\\
0
\end{array}
\right](0,2\tau),
\end{equation*}
which implies that 
\begin{equation*}
\theta^4
\left[
\begin{array}{c}
1 \\
0
\end{array}
\right](0,\tau)
=
4
\theta^2
\left[
\begin{array}{c}
0 \\
0
\end{array}
\right](0,2\tau)
\theta^2
\left[
\begin{array}{c}
1\\
0
\end{array}
\right](0,2\tau).
\end{equation*}
The derivative formulas (\ref{eqn:analogue-0-1/4}) and (\ref{eqn:analogue-0-3/4}) yield 
\begin{equation*}
\left\{
\frac
{
\theta^{\prime} 
\left[
\begin{array}{c}
0 \\
\frac14
\end{array}
\right](0,\tau) 
}
{
\theta
\left[
\begin{array}{c}
0 \\
\frac14
\end{array}
\right](0,\tau)
}
\right\}^2
+
\left\{
\frac
{
\theta^{\prime} 
\left[
\begin{array}{c}
0 \\
\frac34
\end{array}
\right](0,\tau) 
}
{
\theta
\left[
\begin{array}{c}
0 \\
\frac34
\end{array}
\right](0,\tau)
}
\right\}^2
=
\pi^2
\theta^4
\left[
\begin{array}{c}
1 \\
0
\end{array}
\right](0,2\tau)
+
6
\pi^2
\theta^4
\left[
\begin{array}{c}
1 \\
0
\end{array}
\right](0,4\tau). 
\end{equation*}
By the proof of Theorem \ref{thm:4-triangular}, we have 
\begin{align*}
\pi^2
\theta^4
\left[
\begin{array}{c}
1 \\
0
\end{array}
\right](0,2\tau)
+
6
\pi^2
\theta^4
\left[
\begin{array}{c}
1 \\
0
\end{array}
\right](0,4\tau)
=&
-
4\pi i
\frac{d}{d\tau}
\log \frac{\eta^3(2\tau)}{\eta^2(\tau) \eta(4\tau)}
-
6\cdot
4\pi i
\cdot\frac12
\cdot
\frac{d}{d\tau}
\log \frac{\eta^3(4\tau)}{\eta^2(2\tau) \eta(8\tau)}  \\
=&
4\pi i 
\frac{d}{d\tau}
\log
\frac{\eta^2(\tau)\eta^3(2\tau) \eta^3(8\tau)}{\eta^8(4\tau)}, 
\end{align*}
which proves the theorem. 
\end{proof}

\subsection{Applications}
In this subsection, 
for each $m\in\mathbb{N},$ 
we set 
\begin{equation*}
\left(\frac{8}{m}\right) 
=
\begin{cases}
+1, \,\,&\text{if}  \,\,m\equiv \pm1  \,\,(\mathrm{mod} \, 8),  \\
-1, \,\,&\text{if}  \,\,m\equiv \pm3 \,\,(\mathrm{mod} \, 8),  \\
0, \,\,&\text{if}  \,\,m\equiv 0 \,\,(\mathrm{mod} \, 2).  \\
\end{cases}
\end{equation*}

\subsubsection{Preliminary results}

\begin{proposition}
\label{prop:preliminary-1-1/4,3/4}
{\it
For every $\tau\in\mathbb{H}^2,$ we have 
\begin{equation}
\label{eqn:relation-1-1/4,3/4(1)}
\frac
{
\theta^{\prime\prime}
\left[
\begin{array}{c}
1 \\
\frac12
\end{array}
\right]
}
{
\theta
\left[
\begin{array}{c}
1 \\
\frac12
\end{array}
\right]
}
+
2
\frac
{
\theta^{\prime\prime}
\left[
\begin{array}{c}
1 \\
\frac14
\end{array}
\right]
}
{
\theta
\left[
\begin{array}{c}
1 \\
\frac14
\end{array}
\right]
}
-
\frac
{
\theta^{\prime\prime \prime}
\left[
\begin{array}{c}
1 \\
1
\end{array}
\right]
}
{
\theta^{\prime}
\left[
\begin{array}{c}
1 \\
1
\end{array}
\right]
}
+
4
\frac
{
\theta^{\prime} 
\left[
\begin{array}{c}
1 \\
\frac12
\end{array}
\right]
}
{
\theta
\left[
\begin{array}{c}
1 \\
\frac12
\end{array}
\right]
}
\cdot
\frac
{
\theta^{\prime} 
\left[
\begin{array}{c}
1 \\
\frac14
\end{array}
\right]
}
{
\theta
\left[
\begin{array}{c}
1 \\
\frac14
\end{array}
\right]
}
+
2
\left\{
\frac
{
\theta^{\prime} 
\left[
\begin{array}{c}
1 \\
\frac14
\end{array}
\right]
}
{
\theta
\left[
\begin{array}{c}
1 \\
\frac14
\end{array}
\right]
}
\right\}^2=0,
\end{equation}
and 
\begin{equation}
\label{eqn:relation-1-1/4,3/4(2)}
\frac
{
\theta^{\prime\prime}
\left[
\begin{array}{c}
1 \\
\frac12
\end{array}
\right]
}
{
\theta
\left[
\begin{array}{c}
1 \\
\frac12
\end{array}
\right]
}
+
2
\frac
{
\theta^{\prime\prime}
\left[
\begin{array}{c}
1 \\
\frac34
\end{array}
\right]
}
{
\theta
\left[
\begin{array}{c}
1 \\
\frac34
\end{array}
\right]
}
-
\frac
{
\theta^{\prime\prime \prime}
\left[
\begin{array}{c}
1 \\
1
\end{array}
\right]
}
{
\theta^{\prime}
\left[
\begin{array}{c}
1 \\
1
\end{array}
\right]
}
-
4
\frac
{
\theta^{\prime} 
\left[
\begin{array}{c}
1 \\
\frac12
\end{array}
\right]
}
{
\theta
\left[
\begin{array}{c}
1 \\
\frac12
\end{array}
\right]
}
\cdot
\frac
{
\theta^{\prime} 
\left[
\begin{array}{c}
1 \\
\frac34
\end{array}
\right]
}
{
\theta
\left[
\begin{array}{c}
1 \\
\frac34
\end{array}
\right]
}
+
2
\left\{
\frac
{
\theta^{\prime} 
\left[
\begin{array}{c}
1 \\
\frac34
\end{array}
\right]
}
{
\theta
\left[
\begin{array}{c}
1 \\
\frac34
\end{array}
\right]
}
\right\}^2=0.
\end{equation}
}
\end{proposition}

\begin{proof}
Consider the following elliptic functions:
\begin{equation*}
\varphi(z)
=
\frac
{
\theta^2
\left[
\begin{array}{c}
1\\
\frac14
\end{array}
\right](z)
\theta
\left[
\begin{array}{c}
1 \\
\frac12
\end{array}
\right](z) 
}
{
\theta^3
\left[
\begin{array}{c}
1 \\
1
\end{array}
\right](z)
}, \,\,
\psi(z)
=
\frac
{
\theta^2
\left[
\begin{array}{c}
1 \\
\frac34
\end{array}
\right](z)
\theta
\left[
\begin{array}{c}
1 \\
-\frac12
\end{array}
\right](z) 
}
{
\theta^3
\left[
\begin{array}{c}
1 \\
1
\end{array}
\right](z)
}. 
\end{equation*}
In the fundamental parallelogram, 
the pole of $\varphi(z)$ or $\psi(z)$ is $z=0,$ 
which implies that 
$\mathrm{Res} (\varphi(z), 0)=\mathrm{Res} (\psi(z), 0)=0.$ 
The proposition follows from direct calculation. 
\end{proof}

\begin{proposition}
\label{prop:preliminary-0-1/4,3/4}
{\it
For every $\tau\in\mathbb{H}^2,$ we have 
\begin{equation}
\label{eqn:relation-0-1/4,3/4(1)}
\frac
{
\theta^{\prime\prime}
\left[
\begin{array}{c}
1 \\
\frac12
\end{array}
\right]
}
{
\theta
\left[
\begin{array}{c}
1 \\
\frac12
\end{array}
\right]
}
+
2
\frac
{
\theta^{\prime\prime}
\left[
\begin{array}{c}
0 \\
\frac14
\end{array}
\right]
}
{
\theta
\left[
\begin{array}{c}
0 \\
\frac14
\end{array}
\right]
}
-
\frac
{
\theta^{\prime\prime \prime}
\left[
\begin{array}{c}
1 \\
1
\end{array}
\right]
}
{
\theta^{\prime}
\left[
\begin{array}{c}
1 \\
1
\end{array}
\right]
}
+
4
\frac
{
\theta^{\prime} 
\left[
\begin{array}{c}
1 \\
\frac12
\end{array}
\right]
}
{
\theta
\left[
\begin{array}{c}
1 \\
\frac12
\end{array}
\right]
}
\cdot
\frac
{
\theta^{\prime} 
\left[
\begin{array}{c}
0 \\
\frac14
\end{array}
\right]
}
{
\theta
\left[
\begin{array}{c}
0 \\
\frac14
\end{array}
\right]
}
+
2
\left\{
\frac
{
\theta^{\prime} 
\left[
\begin{array}{c}
0 \\
\frac14
\end{array}
\right]
}
{
\theta
\left[
\begin{array}{c}
0 \\
\frac14
\end{array}
\right]
}
\right\}^2=0,
\end{equation}
and 
\begin{equation}
\label{eqn:relation-0-1/4,3/4(2)}
\frac
{
\theta^{\prime\prime}
\left[
\begin{array}{c}
1 \\
\frac12
\end{array}
\right]
}
{
\theta
\left[
\begin{array}{c}
1 \\
\frac12
\end{array}
\right]
}
+
2
\frac
{
\theta^{\prime\prime}
\left[
\begin{array}{c}
0 \\
\frac34
\end{array}
\right]
}
{
\theta
\left[
\begin{array}{c}
0 \\
\frac34
\end{array}
\right]
}
-
\frac
{
\theta^{\prime\prime \prime}
\left[
\begin{array}{c}
1 \\
1
\end{array}
\right]
}
{
\theta^{\prime}
\left[
\begin{array}{c}
1 \\
1
\end{array}
\right]
}
-
4
\frac
{
\theta^{\prime} 
\left[
\begin{array}{c}
1 \\
\frac12
\end{array}
\right]
}
{
\theta
\left[
\begin{array}{c}
1 \\
\frac12
\end{array}
\right]
}
\cdot
\frac
{
\theta^{\prime} 
\left[
\begin{array}{c}
0 \\
\frac34
\end{array}
\right]
}
{
\theta
\left[
\begin{array}{c}
0 \\
\frac34
\end{array}
\right]
}
+
2
\left\{
\frac
{
\theta^{\prime} 
\left[
\begin{array}{c}
0 \\
\frac34
\end{array}
\right]
}
{
\theta
\left[
\begin{array}{c}
0 \\
\frac34
\end{array}
\right]
}
\right\}^2=0.
\end{equation}
}
\end{proposition}

\begin{proof}
Consider the following elliptic functions:
\begin{equation*}
\varphi(z)
=
\frac
{
\theta^2
\left[
\begin{array}{c}
0\\
\frac14
\end{array}
\right](z)
\theta
\left[
\begin{array}{c}
1 \\
\frac12
\end{array}
\right](z) 
}
{
\theta^3
\left[
\begin{array}{c}
1 \\
1
\end{array}
\right](z)
}, \,\,
\psi(z)
=
\frac
{
\theta^2
\left[
\begin{array}{c}
0 \\
\frac34
\end{array}
\right](z)
\theta
\left[
\begin{array}{c}
1 \\
-\frac12
\end{array}
\right](z) 
}
{
\theta^3
\left[
\begin{array}{c}
1 \\
1
\end{array}
\right](z)
}. 
\end{equation*}
In the fundamental parallelogram, 
the pole of $\varphi(z)$ or $\psi(z)$ is $z=0,$ 
which implies that 
$\mathrm{Res} (\varphi(z), 0)=\mathrm{Res} (\psi(z), 0)=0.$ 
The proposition follows from direct calculation. 
\end{proof}

\subsubsection{On $x^2+y^2+2z^2+2w^2$}

\begin{theorem}
\label{thm:1122}
{\it
For each $n\in\mathbb{N},$ set 
\begin{equation*}
S_{1,1,2,2}(n)=
\sharp
\left\{
(x,y,z,w)\in\mathbb{Z}^4 \, | \,
x^2+y^2+2z^2+2w^2=n
\right\}. 
\end{equation*}
Then, we have 
\begin{equation*}
S_{1,1,2,2}(n)
=
4\sigma(n)-4\sigma(n/2)+8\sigma(n/4)-32\sigma(n/8). 
\end{equation*}
}
\end{theorem}

\begin{proof}
Set $q=\exp(2\pi i \tau).$ 
Summing both sides of equations (\ref{eqn:relation-1-1/4,3/4(1)}) and (\ref{eqn:relation-1-1/4,3/4(2)}) yields 
\begin{align*}
&
\frac
{
\theta^{\prime\prime}
\left[
\begin{array}{c}
1 \\
\frac12
\end{array}
\right]
}
{
\theta
\left[
\begin{array}{c}
1 \\
\frac12
\end{array}
\right]
}
+
\frac
{
\theta^{\prime\prime}
\left[
\begin{array}{c}
1 \\
\frac14
\end{array}
\right]
}
{
\theta
\left[
\begin{array}{c}
1 \\
\frac14
\end{array}
\right]
}
+
\frac
{
\theta^{\prime\prime}
\left[
\begin{array}{c}
1 \\
\frac34
\end{array}
\right]
}
{
\theta
\left[
\begin{array}{c}
1 \\
\frac34
\end{array}
\right]
}
-
\frac
{
\theta^{\prime\prime \prime}
\left[
\begin{array}{c}
1 \\
1
\end{array}
\right]
}
{
\theta^{\prime}
\left[
\begin{array}{c}
1 \\
1
\end{array}
\right]
}  \\
&
+
2
\frac
{
\theta^{\prime} 
\left[
\begin{array}{c}
1 \\
\frac12
\end{array}
\right]
}
{
\theta
\left[
\begin{array}{c}
1 \\
\frac12
\end{array}
\right]
}
\cdot
\left\{
\frac
{
\theta^{\prime} 
\left[
\begin{array}{c}
1 \\
\frac14
\end{array}
\right]
}
{
\theta
\left[
\begin{array}{c}
1 \\
\frac14
\end{array}
\right]
}
-
\frac
{
\theta^{\prime} 
\left[
\begin{array}{c}
1 \\
\frac34
\end{array}
\right]
}
{
\theta
\left[
\begin{array}{c}
1 \\
\frac34
\end{array}
\right]
}
\right\}
+
\left[
\left\{
\frac
{
\theta^{\prime} 
\left[
\begin{array}{c}
1 \\
\frac14
\end{array}
\right]
}
{
\theta
\left[
\begin{array}{c}
1 \\
\frac14
\end{array}
\right]
}
\right\}^2
+
\left\{
\frac
{
\theta^{\prime} 
\left[
\begin{array}{c}
1 \\
\frac34
\end{array}
\right]
}
{
\theta
\left[
\begin{array}{c}
1 \\
\frac34
\end{array}
\right]
}
\right\}^2
\right]=0.
\end{align*}
The derivative formulas (\ref{eqn:analogue-1-1/2}), (\ref{eqn:analogue-1-1/4}) and (\ref{eqn:analogue-1-3/4}) imply that 
\begin{align*}
&
\frac
{
\theta^{\prime\prime}
\left[
\begin{array}{c}
1 \\
\frac12
\end{array}
\right]
}
{
\theta
\left[
\begin{array}{c}
1 \\
\frac12
\end{array}
\right]
}
+
\frac
{
\theta^{\prime\prime}
\left[
\begin{array}{c}
1 \\
\frac14
\end{array}
\right]
}
{
\theta
\left[
\begin{array}{c}
1 \\
\frac14
\end{array}
\right]
}
+
\frac
{
\theta^{\prime\prime}
\left[
\begin{array}{c}
1 \\
\frac34
\end{array}
\right]
}
{
\theta
\left[
\begin{array}{c}
1 \\
\frac34
\end{array}
\right]
}
-
\frac
{
\theta^{\prime\prime \prime}
\left[
\begin{array}{c}
1 \\
1
\end{array}
\right]
}
{
\theta^{\prime}
\left[
\begin{array}{c}
1 \\
1
\end{array}
\right]
}
+
\left[
\left\{
\frac
{
\theta^{\prime} 
\left[
\begin{array}{c}
1 \\
\frac14
\end{array}
\right]
}
{
\theta
\left[
\begin{array}{c}
1 \\
\frac14
\end{array}
\right]
}
\right\}^2
+
\left\{
\frac
{
\theta^{\prime} 
\left[
\begin{array}{c}
1 \\
\frac34
\end{array}
\right]
}
{
\theta
\left[
\begin{array}{c}
1 \\
\frac34
\end{array}
\right]
}
\right\}^2
\right] \\
=&
4\pi^2
\theta^2
\left[
\begin{array}{c}
0 \\
0
\end{array}
\right](0,2\tau)
\theta^2
\left[
\begin{array}{c}
0 \\
0
\end{array}
\right](0,4\tau). 
\end{align*}
The heat equation (\ref{eqn:heat}) and Jacobi's triple product identity (\ref{eqn:Jacobi-triple}) shows that 
\begin{equation*}
4\pi i \frac{d}{d\tau} \log \frac{\eta(8\tau)}{\eta(2\tau)}+
\left[
\left\{
\frac
{
\theta^{\prime} 
\left[
\begin{array}{c}
1 \\
\frac14
\end{array}
\right]
}
{
\theta
\left[
\begin{array}{c}
1 \\
\frac14
\end{array}
\right]
}
\right\}^2
+
\left\{
\frac
{
\theta^{\prime} 
\left[
\begin{array}{c}
1 \\
\frac34
\end{array}
\right]
}
{
\theta
\left[
\begin{array}{c}
1 \\
\frac34
\end{array}
\right]
}
\right\}^2
\right] 
=
4\pi^2
\theta^2
\left[
\begin{array}{c}
0 \\
0
\end{array}
\right](0,2\tau)
\theta^2
\left[
\begin{array}{c}
0 \\
0
\end{array}
\right](0,4\tau). 
\end{equation*}
Theorem \ref{thm:1-1/4,3/4} yields 
\begin{equation}
\label{eqn:1122-log-derivative}
4\pi i \frac{d}{d\tau} \log 
\frac{\eta^2(\tau) \eta(4\tau)}{\eta(2\tau) \eta^2(8\tau)}
=
4\pi^2
\theta^2
\left[
\begin{array}{c}
0 \\
0
\end{array}
\right](0,2\tau)
\theta^2
\left[
\begin{array}{c}
0 \\
0
\end{array}
\right](0,4\tau). 
\end{equation}
The theorem can be obtained by considering that 
\begin{equation*}
\theta^2
\left[
\begin{array}{c}
0 \\
0
\end{array}
\right](0,2\tau)
\theta^2
\left[
\begin{array}{c}
0 \\
0
\end{array}
\right](0,4\tau)
=
\left(
\sum_{n\in\mathbb{Z}} q^{n^2}
\right)^2
\left(
\sum_{n\in\mathbb{Z}} q^{2n^2}
\right)^2
=
\sum_{x,y,z,w\in\mathbb{Z}}
q^{x^2+y^2+2z^2+2w^2}. 
\end{equation*}
\end{proof}

\subsubsection{On $x^2+2y^2+4t_z+4t_w$}

\begin{theorem}
\label{thm:x^2+2y^2+4t_z+4t_w}
{\it
For each $n\in\mathbb{N}_0,$ set 
\begin{equation*}
M_{1,2\textrm{-}4,4}(n):=
\sharp
\left\{
(x,y,z,w)\in\mathbb{Z}^4 \, | \,
x^2+2y^2+4t_z+4t_w=n
\right\}. 
\end{equation*}
Then, we have 
\begin{equation*}
M_{1,2\textrm{-}4,4}(n)=
4\sum_{d|n+1} \frac{n+1}{d} \left( \frac{8}{d} \right).
\end{equation*}
}
\end{theorem}

\begin{proof}
Set $q=\exp(2\pi i \tau).$ 
By Lemma \ref{lem:Farkas-Kra}, we first note that
\begin{equation*}
\theta^2
\left[
\begin{array}{c}
0 \\
0
\end{array}
\right](0,\tau)
=
\theta^2
\left[
\begin{array}{c}
0 \\
0
\end{array}
\right](0,2\tau)
+
\theta^2
\left[
\begin{array}{c}
1\\
0
\end{array}
\right](0,2\tau). 
\end{equation*}
Subtracting both sides of equations (\ref{eqn:relation-1-1/4,3/4(1)}) and (\ref{eqn:relation-1-1/4,3/4(2)}) yields 
\begin{equation*}
\frac
{
\theta^{\prime\prime}
\left[
\begin{array}{c}
1 \\
\frac14
\end{array}
\right]
}
{
\theta
\left[
\begin{array}{c}
1 \\
\frac14
\end{array}
\right]
}
-
\frac
{
\theta^{\prime\prime}
\left[
\begin{array}{c}
1 \\
\frac34
\end{array}
\right]
}
{
\theta
\left[
\begin{array}{c}
1 \\
\frac34
\end{array}
\right]
}
=
-4\sqrt{2} \pi^2
\theta
\left[
\begin{array}{c}
0 \\
0
\end{array}
\right](0,2\tau)
\theta
\left[
\begin{array}{c}
0 \\
0
\end{array}
\right](0,4\tau)
\theta^2
\left[
\begin{array}{c}
1 \\
0
\end{array}
\right](0,4\tau). 
\end{equation*}
Jacobi's triple product identity (\ref{eqn:Jacobi-triple}) yields 
\begin{equation*}
\frac
{
\theta^{\prime\prime}
\left[
\begin{array}{c}
1 \\
\frac14
\end{array}
\right]
}
{
\theta
\left[
\begin{array}{c}
1 \\
\frac14
\end{array}
\right]
}
-
\frac
{
\theta^{\prime\prime}
\left[
\begin{array}{c}
1 \\
\frac34
\end{array}
\right]
}
{
\theta
\left[
\begin{array}{c}
1 \\
\frac34
\end{array}
\right]
}
=
-16\sqrt{2}\pi^2\sum_{n=1}^{\infty} 
\left(
\sum_{d|n} \frac{n}{d} \left(\frac{8}{d} \right)
\right)
q^n. 
\end{equation*}
From the definition, it follows that 
\begin{align*}
\theta
\left[
\begin{array}{c}
0 \\
0
\end{array}
\right](0,2\tau)
\theta
\left[
\begin{array}{c}
0 \\
0
\end{array}
\right](0,4\tau)
\theta^2
\left[
\begin{array}{c}
1 \\
0
\end{array}
\right](0,4\tau)
=&
\left(
\sum_{n\in\mathbb{Z}} q^{n^2}
\right)
\left(
\sum_{n\in\mathbb{Z}} q^{2n^2}
\right)
\left(
q^{\frac12}
\sum_{n\in\mathbb{Z}} q^{4\cdot\frac{n(n+1)}{2}}
\right)^2  \\
=&
\sum_{n=0}^{\infty} 
M_{1,2\textrm{-}4,4}(n) q^{n+1},
\end{align*}
which proves the theorem. 
\end{proof}

\subsubsection{On $x^2+y^2+4t_z+4t_w$}

\begin{theorem}
\label{thm:$x^2+y^2+4t_z+4t_w}
{\it
For each $n\in\mathbb{N}_0,$ set 
\begin{equation*}
M_{1,1\textrm{-}4,4}(n):=
\sharp
\left\{
(x,y,z,w)\in\mathbb{Z}^4 \, | \,
x^2+y^2+4t_z+4t_w=n
\right\}. 
\end{equation*}
Then, we have 
\begin{equation*}
M_{1,1\textrm{-}4,4}(n)=
4
\left(
\sigma(n+1)+\sigma \left(\frac{n+1}{2}\right)-10\sigma\left(\frac{n+1}{4}\right)+8\sigma \left(\frac{n+1}{8}\right)
\right). 
\end{equation*}
}
\end{theorem}

\begin{proof}
Set $q=\exp(2\pi i \tau).$ 
Summing both sides of equations (\ref{eqn:relation-0-1/4,3/4(1)}) and (\ref{eqn:relation-0-1/4,3/4(2)}) yields 
\begin{align*}
&
\frac
{
\theta^{\prime\prime}
\left[
\begin{array}{c}
1 \\
\frac12
\end{array}
\right]
}
{
\theta
\left[
\begin{array}{c}
1 \\
\frac12
\end{array}
\right]
}
+
\frac
{
\theta^{\prime\prime}
\left[
\begin{array}{c}
0 \\
\frac14
\end{array}
\right]
}
{
\theta
\left[
\begin{array}{c}
0 \\
\frac14
\end{array}
\right]
}
+
\frac
{
\theta^{\prime\prime}
\left[
\begin{array}{c}
0 \\
\frac34
\end{array}
\right]
}
{
\theta
\left[
\begin{array}{c}
0 \\
\frac34
\end{array}
\right]
}
-
\frac
{
\theta^{\prime\prime \prime}
\left[
\begin{array}{c}
1 \\
1
\end{array}
\right]
}
{
\theta^{\prime}
\left[
\begin{array}{c}
1 \\
1
\end{array}
\right]
}  \\
&
+
2
\frac
{
\theta^{\prime} 
\left[
\begin{array}{c}
1 \\
\frac12
\end{array}
\right]
}
{
\theta
\left[
\begin{array}{c}
1 \\
\frac12
\end{array}
\right]
}
\cdot
\left\{
\frac
{
\theta^{\prime} 
\left[
\begin{array}{c}
0 \\
\frac14
\end{array}
\right]
}
{
\theta
\left[
\begin{array}{c}
0 \\
\frac14
\end{array}
\right]
}
-
\frac
{
\theta^{\prime} 
\left[
\begin{array}{c}
0 \\
\frac34
\end{array}
\right]
}
{
\theta
\left[
\begin{array}{c}
0 \\
\frac34
\end{array}
\right]
}
\right\}
+
\left[
\left\{
\frac
{
\theta^{\prime} 
\left[
\begin{array}{c}
0 \\
\frac14
\end{array}
\right]
}
{
\theta
\left[
\begin{array}{c}
0 \\
\frac14
\end{array}
\right]
}
\right\}^2
+
\left\{
\frac
{
\theta^{\prime} 
\left[
\begin{array}{c}
0 \\
\frac34
\end{array}
\right]
}
{
\theta
\left[
\begin{array}{c}
0 \\
\frac34
\end{array}
\right]
}
\right\}^2
\right]=0.
\end{align*}
The derivative formulas (\ref{eqn:analogue-1-1/2}), (\ref{eqn:analogue-0-1/4}) and (\ref{eqn:analogue-0-3/4}) imply that 
\begin{align*}
&
\frac
{
\theta^{\prime\prime}
\left[
\begin{array}{c}
1 \\
\frac12
\end{array}
\right]
}
{
\theta
\left[
\begin{array}{c}
1 \\
\frac12
\end{array}
\right]
}
+
\frac
{
\theta^{\prime\prime}
\left[
\begin{array}{c}
0 \\
\frac14
\end{array}
\right]
}
{
\theta
\left[
\begin{array}{c}
0 \\
\frac14
\end{array}
\right]
}
+
\frac
{
\theta^{\prime\prime}
\left[
\begin{array}{c}
0 \\
\frac34
\end{array}
\right]
}
{
\theta
\left[
\begin{array}{c}
0 \\
\frac34
\end{array}
\right]
}
-
\frac
{
\theta^{\prime\prime \prime}
\left[
\begin{array}{c}
1 \\
1
\end{array}
\right]
}
{
\theta^{\prime}
\left[
\begin{array}{c}
1 \\
1
\end{array}
\right]
}
+
\left[
\left\{
\frac
{
\theta^{\prime} 
\left[
\begin{array}{c}
0 \\
\frac14
\end{array}
\right]
}
{
\theta
\left[
\begin{array}{c}
0 \\
\frac14
\end{array}
\right]
}
\right\}^2
+
\left\{
\frac
{
\theta^{\prime} 
\left[
\begin{array}{c}
0 \\
\frac34
\end{array}
\right]
}
{
\theta
\left[
\begin{array}{c}
0 \\
\frac34
\end{array}
\right]
}
\right\}^2
\right] \\
=&
4\pi^2
\theta^2
\left[
\begin{array}{c}
0 \\
0
\end{array}
\right](0,2\tau)
\theta^2
\left[
\begin{array}{c}
1 \\
0
\end{array}
\right](0,4\tau). 
\end{align*}
The heat equation (\ref{eqn:heat}) and Jacobi's triple product identity (\ref{eqn:Jacobi-triple}) shows that 
\begin{equation*}
4\pi i \frac{d}{d\tau} \log \frac{\eta^3(4\tau)}{\eta^2(2\tau) \eta(8\tau)}+
\left[
\left\{
\frac
{
\theta^{\prime} 
\left[
\begin{array}{c}
0 \\
\frac14
\end{array}
\right]
}
{
\theta
\left[
\begin{array}{c}
0 \\
\frac14
\end{array}
\right]
}
\right\}^2
+
\left\{
\frac
{
\theta^{\prime} 
\left[
\begin{array}{c}
0 \\
\frac34
\end{array}
\right]
}
{
\theta
\left[
\begin{array}{c}
0 \\
\frac34
\end{array}
\right]
}
\right\}^2
\right] 
=
4\pi^2
\theta^2
\left[
\begin{array}{c}
0 \\
0
\end{array}
\right](0,2\tau)
\theta^2
\left[
\begin{array}{c}
1 \\
0
\end{array}
\right](0,4\tau). 
\end{equation*}
Theorem \ref{thm:0-1/4,3/4} yields 
\begin{equation}
\label{eqn:11-44-log-derivative}
4\pi i \frac{d}{d\tau} \log 
\frac{\eta^2(\tau) \eta(2\tau) \eta^2(8\tau)}{\eta^5(4\tau)}
=
4\pi^2
\theta^2
\left[
\begin{array}{c}
0 \\
0
\end{array}
\right](0,2\tau)
\theta^2
\left[
\begin{array}{c}
1 \\
0
\end{array}
\right](0,4\tau). 
\end{equation}
The theorem can be obtained by considering that 
\begin{equation*}
\theta^2
\left[
\begin{array}{c}
0 \\
0
\end{array}
\right](0,2\tau)
\theta^2
\left[
\begin{array}{c}
1 \\
0
\end{array}
\right](0,4\tau)
=
\left(
\sum_{n\in\mathbb{Z}} q^{n^2}
\right)^2
\left(
q^{\frac12}
\sum_{n\in\mathbb{Z}} q^{4\cdot \frac{n(n+1)}{2}}
\right)^2
=
\sum_{x,y,z,w\in\mathbb{Z}}
q^{x^2+y^2+4t_z+4t_w+1}. 
\end{equation*}
\end{proof}

\subsubsection{On $x^2+2y^2+2z^2+4t_w$}

\begin{theorem}
\label{thm:$x^2+2y^2+2z^2+4t_w}
{\it
For each $n\in\mathbb{N}_0,$ set 
\begin{equation*}
M_{1,2,2\textrm{-}4}(n):=
\sharp
\left\{
(x,y,z,w)\in\mathbb{Z}^4 \, | \,
x^2+2y^2+2z^2+4t_w=n
\right\}. 
\end{equation*}
Then, we have 
\begin{equation*}
M_{1,2,2\textrm{-}4}(n)=
2\sum_{d|2n+1} \frac{2n+1}{d} \left( \frac{8}{d} \right).
\end{equation*}
}
\end{theorem}

\begin{proof}
Set $x=\exp(\pi i \tau)$ and $q=\exp(2\pi i \tau).$ 
By Lemma \ref{lem:Farkas-Kra}, we first note that
\begin{equation*}
\theta^2
\left[
\begin{array}{c}
0 \\
0
\end{array}
\right](0,\tau)
=
\theta^2
\left[
\begin{array}{c}
0 \\
0
\end{array}
\right](0,2\tau)
+
\theta^2
\left[
\begin{array}{c}
1\\
0
\end{array}
\right](0,2\tau). 
\end{equation*}
Subtracting both sides of equations (\ref{eqn:relation-0-1/4,3/4(1)}) and (\ref{eqn:relation-0-1/4,3/4(2)}) yields 
\begin{equation*}
\frac
{
\theta^{\prime\prime}
\left[
\begin{array}{c}
0 \\
\frac14
\end{array}
\right]
}
{
\theta
\left[
\begin{array}{c}
0 \\
\frac14
\end{array}
\right]
}
-
\frac
{
\theta^{\prime\prime}
\left[
\begin{array}{c}
0 \\
\frac34
\end{array}
\right]
}
{
\theta
\left[
\begin{array}{c}
0 \\
\frac34
\end{array}
\right]
}
=
-4\sqrt{2} \pi^2
\theta
\left[
\begin{array}{c}
0 \\
0
\end{array}
\right](0,2\tau)
\theta^2
\left[
\begin{array}{c}
0 \\
0
\end{array}
\right](0,4\tau)
\theta
\left[
\begin{array}{c}
1 \\
0
\end{array}
\right](0,4\tau). 
\end{equation*}
Jacobi's triple product identity (\ref{eqn:Jacobi-triple}) yields 
\begin{equation*}
\frac
{
\theta^{\prime\prime}
\left[
\begin{array}{c}
0 \\
\frac14
\end{array}
\right]
}
{
\theta
\left[
\begin{array}{c}
0 \\
\frac14
\end{array}
\right]
}
-
\frac
{
\theta^{\prime\prime}
\left[
\begin{array}{c}
0 \\
\frac34
\end{array}
\right]
}
{
\theta
\left[
\begin{array}{c}
0 \\
\frac34
\end{array}
\right]
}
=
-8\sqrt{2}\pi^2\sum_{n=0}^{\infty} 
\left(
\sum_{d|2n+1} \frac{2n+1}{d}\left(\frac{8}{d} \right)
\right)
x^{2n+1}.  
\end{equation*}
From the definition, it follows that 
\begin{align*}
\theta
\left[
\begin{array}{c}
0 \\
0
\end{array}
\right](0,2\tau)
\theta^2
\left[
\begin{array}{c}
0 \\
0
\end{array}
\right](0,4\tau)
\theta
\left[
\begin{array}{c}
1 \\
0
\end{array}
\right](0,4\tau)
=&
\left(
\sum_{n\in\mathbb{Z}} q^{n^2}
\right)
\left(
\sum_{n\in\mathbb{Z}} q^{2n^2}
\right)^2
\left(
q^{\frac12}
\sum_{n\in\mathbb{Z}} q^{4\cdot\frac{n(n+1)}{2}}
\right)  \\
=&
\sum_{n=0}^{\infty} 
M_{1,2,2\textrm{-}4}(n) x^{2n+1},
\end{align*}
which proves the theorem. 
\end{proof}

\subsection{More applications}
In this subsection, 
for each $m\in\mathbb{N},$ 
we set
\begin{equation*}
\left(\frac{8}{m}\right) 
=
\begin{cases}
+1, \,\,&\text{if}  \,\,m\equiv \pm1  \,\,(\mathrm{mod} \, 8),  \\
-1, \,\,&\text{if}  \,\,m\equiv \pm3 \,\,(\mathrm{mod} \, 8),  \\
0, \,\,&\text{if}  \,\,m\equiv 0 \,\,(\mathrm{mod} \, 2).  \\
\end{cases}
\end{equation*}
Moreover we have the following lemma:
\begin{lemma}
\label{lem:first-order-deri-square}
{\it
For every $(z,\tau)\in\mathbb{C}\times\mathbb{H}^2,$ 
we have 
\begin{equation*}
\left\{
\frac{
\theta^{\prime}
\left[
\begin{array}{c}
\epsilon \\
\epsilon^{\prime}
\end{array}
\right](z)
}
{
\theta
\left[
\begin{array}{c}
\epsilon \\
\epsilon^{\prime}
\end{array}
\right](z)
}
\right\}^2
=
\frac{
\theta^{\prime \prime}
\left[
\begin{array}{c}
\epsilon \\
\epsilon^{\prime}
\end{array}
\right](z)
}
{
\theta
\left[
\begin{array}{c}
\epsilon \\
\epsilon^{\prime}
\end{array}
\right](z)
}
-
\frac
{
d^2
}
{
dz^2
}
\log
\theta
\left[
\begin{array}{c}
\epsilon \\
\epsilon^{\prime}
\end{array}
\right](z). 
\end{equation*}
}
\end{lemma}

\begin{proof}
The lemma can be proved by direct calculation. 
\end{proof}

\subsubsection{On $x^2+y^2+z^2+2w^2,$ and $x^2+2y^2+2z^2+2w^2$}

\begin{theorem}
\label{thm:x^2+y^2+z^2+2w^2}
{\it
For each $n\in\mathbb{N},$ set 
\begin{equation*}
S_{1,1,1,2}(n):=
\sharp
\left\{
(x,y,z,w)\in\mathbb{Z}^4 \, | \,
x^2+y^2+z^2+2w^2=n
\right\}. 
\end{equation*}
Then, we have 
\begin{equation*}
S_{1,1,1,2}(n)=
8
\sum_{d|n} \frac{n}{d} \left( \frac{8}{d} \right)
-
2
\sum_{d|n}d \left( \frac{8}{d} \right). 
\end{equation*}
}
\end{theorem}

\begin{proof}
We first note that for each $n\in\mathbb{N},$
\begin{equation*}
\sum_{d|n, d:odd} d=
\sigma(n)-2\sigma \left(\frac{n}{2}\right). 
\end{equation*}
The derivatve formulas (\ref{eqn:analogue-0-1/4}), (\ref{eqn:analogue-0-3/4}), and  equation (\ref{eqn:relation-1-1/4,3/4(1)}) yield
\begin{align*}
&
4\sqrt{2} \pi^2 
\theta^3
\left[
\begin{array}{c}
0\\
0
\end{array}
\right](0,2\tau)
\theta
\left[
\begin{array}{c}
0\\
0
\end{array}
\right](0,4\tau) \\
=&
-4\pi i 
\frac{d}{d\tau}
\log
\frac
{
\theta
\left[
\begin{array}{c}
1 \\
\frac12
\end{array}
\right]
\theta^4
\left[
\begin{array}{c}
1 \\
\frac14
\end{array}
\right]
\theta^4
\left[
\begin{array}{c}
1 \\
\frac34
\end{array}
\right]
}
{
\theta^{\prime}
\left[
\begin{array}{c}
1 \\
1
\end{array}
\right]
\theta^4
\left[
\begin{array}{c}
1 \\
\frac34
\end{array}
\right]
}
+
4 \pi^2 
\theta^2
\left[
\begin{array}{c}
0\\
0
\end{array}
\right](0,2\tau)
\theta^2
\left[
\begin{array}{c}
0\\
0
\end{array}
\right](0,4\tau)  \\
&\hspace{30mm}
+
2
\left.
\frac{d^2}{dz^2}
\log
\theta
\left[
\begin{array}{c}
1 \\
\frac14
\end{array}
\right]
(z)
\right|_{z=0}. 
\end{align*}
The theorem follows from equation (\ref{eqn:1122-log-derivative}) and Jacobi's triple product identity (\ref{eqn:Jacobi-triple}).
\end{proof}

\begin{theorem}
\label{thm:x^2+2y^2+2z^2+2w^2}
{\it
For each $n\in\mathbb{N},$ set 
\begin{equation*}
S_{1,2,2,2}(n):=
\sharp
\left\{
(x,y,z,w)\in\mathbb{Z}^4 \, | \,
x^2+2y^2+2z^2+2w^2=n
\right\}. 
\end{equation*}
Then, we have 
\begin{equation*}
S_{1,2,2,2}(n)=
4
\sum_{d|n} \frac{n}{d} \left( \frac{8}{d} \right)
-
2
\sum_{d|n}d \left( \frac{8}{d} \right). 
\end{equation*}
}
\end{theorem}

\begin{proof}
The derivatve formulas (\ref{eqn:analogue-0-1/4}), (\ref{eqn:analogue-0-3/4}), and  equation (\ref{eqn:relation-1-1/4,3/4(1)}) yield 
\begin{align*}
&
-4\sqrt{2} \pi^2 
\theta^3
\left[
\begin{array}{c}
0 \\
0
\end{array}
\right](0,2\tau)
\theta
\left[
\begin{array}{c}
0 \\
0
\end{array}
\right](0,4\tau)
+
4\sqrt{2} \pi^2 
\theta
\left[
\begin{array}{c}
0 \\
0
\end{array}
\right](0,2\tau)
\theta^3
\left[
\begin{array}{c}
0 \\
0
\end{array}
\right](0,4\tau)  \\
=&
4\pi i 
\frac{d}{d\tau}
\log
\frac
{
\theta
\left[
\begin{array}{c}
1 \\
\frac12
\end{array}
\right]
\theta^2
\left[
\begin{array}{c}
1 \\
\frac14
\end{array}
\right]
}
{
\theta^{\prime}
\left[
\begin{array}{c}
1 \\
1
\end{array}
\right]
}
+
2\pi^2
\theta^4
\left[
\begin{array}{c}
0 \\
0
\end{array}
\right](0,4\tau). 
\end{align*}
The theorem follows from Theorems \ref{thm:4-squares} and \ref{thm:x^2+y^2+z^2+2w^2}. 
\end{proof}

\subsubsection{On $x^2+y^2+z^2+4t_w,$ and $x^2+4t_y+4t_z+4t_w$}

\begin{theorem}
\label{thm:x^2+y^2+z^2+4t_w}
{\it
For each $n\in\mathbb{N}_0,$ set 
\begin{equation*}
M_{1,1,1\textrm{-}4}(n):=
\sharp
\left\{
(x,y,z,w)\in\mathbb{Z}^4 \, | \,
x^2+y^2+z^2+4t_w=n
\right\}. 
\end{equation*}
Then, we have 
\begin{equation*}
M_{1,1,1\textrm{-}4}(n)=
4
\sum_{d|2n+1} \frac{2n+1}{d} \left( \frac{8}{d} \right)
-
2
\sum_{d|2n+1}d \left( \frac{8}{d} \right). 
\end{equation*}
}
\end{theorem}

\begin{proof}
We first note that for each $n\in\mathbb{N},$
\begin{equation*}
3\sigma(n)-\sigma(2n)-2\sigma \left(\frac{n}{2}\right)=0. 
\end{equation*}
The derivatve formulas (\ref{eqn:analogue-0-1/4}), (\ref{eqn:analogue-0-3/4}), and  equation (\ref{eqn:relation-0-1/4,3/4(1)}) yield
\begin{align*}
&
4\sqrt{2} \pi^2 
\theta^3
\left[
\begin{array}{c}
0\\
0
\end{array}
\right](0,2\tau)
\theta
\left[
\begin{array}{c}
1\\
0
\end{array}
\right](0,4\tau) \\
=&
-4\pi i 
\frac{d}{d\tau}
\log
\frac
{
\theta
\left[
\begin{array}{c}
1 \\
\frac12
\end{array}
\right]
\theta^4
\left[
\begin{array}{c}
0 \\
\frac14
\end{array}
\right]
\theta^4
\left[
\begin{array}{c}
0 \\
\frac34
\end{array}
\right]
}
{
\theta^{\prime}
\left[
\begin{array}{c}
1 \\
1
\end{array}
\right]
\theta^4
\left[
\begin{array}{c}
0 \\
\frac34
\end{array}
\right]
}
+
4 \pi^2 
\theta^2
\left[
\begin{array}{c}
0\\
0
\end{array}
\right](0,2\tau)
\theta^2
\left[
\begin{array}{c}
1\\
0
\end{array}
\right](0,4\tau)  \\
&\hspace{30mm}
+
2
\left.
\frac{d^2}{dz^2}
\log
\theta
\left[
\begin{array}{c}
0 \\
\frac14
\end{array}
\right]
(z)
\right|_{z=0}. 
\end{align*}
The theorem follows from equation (\ref{eqn:11-44-log-derivative}) and Jacobi's triple product identity (\ref{eqn:Jacobi-triple}).
\end{proof}

\begin{theorem}
\label{thm:x^2+4t_y+4t_z+4t_w}
{\it
For each $n\in\mathbb{N}_0,$ set 
\begin{equation*}
M_{1\textrm{-}4,4,4}(n):=
\sharp
\left\{
(x,y,z,w)\in\mathbb{Z}^4 \, | \,
x^2+4t_y+4t_z+4t_w=n
\right\}. 
\end{equation*}
Then, we have 
\begin{equation*}
M_{1\textrm{-}4,4,4}(n)=
2
\sum_{d|2n+3} \frac{2n+3}{d} \left( \frac{8}{d} \right)
-
2
\sum_{d|2n+3}d \left( \frac{8}{d} \right). 
\end{equation*}
}
\end{theorem}

\begin{proof}
The derivatve formulas (\ref{eqn:analogue-0-1/4}), (\ref{eqn:analogue-0-3/4}), and  equation (\ref{eqn:relation-0-1/4,3/4(1)}) yield 
\begin{align*}
&
-4\sqrt{2} \pi^2 
\theta^3
\left[
\begin{array}{c}
0 \\
0
\end{array}
\right](0,2\tau)
\theta
\left[
\begin{array}{c}
1 \\
0
\end{array}
\right](0,4\tau)
+
4\sqrt{2} \pi^2 
\theta
\left[
\begin{array}{c}
0 \\
0
\end{array}
\right](0,2\tau)
\theta^3
\left[
\begin{array}{c}
1 \\
0
\end{array}
\right](0,4\tau)  \\
=&
4\pi i 
\frac{d}{d\tau}
\log
\frac
{
\theta
\left[
\begin{array}{c}
1 \\
\frac12
\end{array}
\right]
\theta^2
\left[
\begin{array}{c}
0 \\
\frac14
\end{array}
\right]
}
{
\theta^{\prime}
\left[
\begin{array}{c}
1 \\
1
\end{array}
\right]
}
+
2\pi^2
\theta^4
\left[
\begin{array}{c}
1 \\
0
\end{array}
\right](0,4\tau). 
\end{align*}
The theorem follows from Theorems \ref{thm:4-triangular} and \ref{thm:x^2+2y^2+4t_z+4t_w}. 
\end{proof}




\end{document}